\documentclass[reqno]{amsart}
\usepackage[margin=1in]{geometry}
\usepackage{tikz-cd}
\usepackage{amscd, amsfonts, amsmath,amssymb,amsthm, mathtools} 
\usepackage{tabularx,ltablex}
\usepackage{todonotes}
\usepackage[all,color]{xy}
\usepackage[colorlinks=false]{hyperref}
\usepackage{cleveref}
\usepackage{float}
\usepackage{changes}
\usepackage{mathrsfs} 
\usepackage{amsmath,latexsym, bbold}
\usepackage{endnotes}
\usepackage{ulem}
\usepackage{aligned-overset}
\usepackage{comment}

\allowdisplaybreaks

\newtheorem{introques}{Question}
\newtheorem{introdefn}[introques]{Definition}
\newtheorem{introdethm}[introques]{Theorem}
\newtheorem{introdeprop}[introques]{Proposition}

\newtheorem{thm}{Theorem}[subsection]
\newtheorem{defn}[thm]{Definition}
\newtheorem{Ex}[thm]{Example} 
\newtheorem{lemma}[thm]{Lemma} 
\newtheorem{proposition}[thm]{Proposition} 
\newtheorem{remark}[thm]{Remark} 
\newtheorem{Cor}[thm]{Corollary}

\DeclareMathOperator{\GL}{GL}

\DeclareMathOperator{\id}{id}

\DeclareMathOperator{\Hom}{Hom}

\newcommand{\uaut}{\underline{\rm aut}}

\newcommand{\unit}{\mathbb{1}}

\def\ob{\operatorname {ob}}

\def\bdt{\Delta}

\def\ot{\otimes}

\newcommand\replace[1]{}

\DeclareMathOperator\ev{{\operatorname{ev}}}
\DeclareMathOperator\coev{\operatorname{coev}}

\makeatletter
\def\namelabel#1#2{\@bsphack
  \protected@write\@auxout{}%
         {\string\newlabel{#1.nme}{{#2}{#2}}}%
  \@esphack}

\makeatother

\def\Hom{{\mbox{\rm Hom}}}

\def\GL{{\mbox{\rm GL}}}

\allowdisplaybreaks
\numberwithin{equation}{section}

\def\1{\mathbb{1}}

\def\ob{\operatorname{ob}}

\newcommand{\kk}{\Bbbk}

\newcommand{\mc}{\mathcal}

\def\ra{\rightarrow}

\def\t{\text}
\def\it{\textit}

\def\vps{\varepsilon}

\numberwithin{equation}{section}

\title[A cogroupoid associated to preregular forms]{A cogroupoid associated to preregular forms}

\author[Huang]{Hongdi Huang}
\address{(Huang) Department of Mathematics, Rice University, Houston, TX 77005, U.S.A.}
\email{h237huan@rice.edu}

\author[Nguyen]{Van C. Nguyen}
\address{(Nguyen) Department of Mathematics, United States Naval Academy, Annapolis, MD 21402, U.S.A.}
\email{vnguyen@usna.edu}

\author[Ure]{Charlotte Ure}
\address{(Ure) Department of Mathematics, University of Virginia, Charlottesville, VA 22904, U.S.A.}
\email{cu9da@virginia.edu}

\author[Vashaw]{Kent B. Vashaw}
\address{(Vashaw) Department of Mathematics,
Massachusetts Institute of Technology,
Cambridge, MA 02139, U.S.A.}
\email{kentv@mit.edu}

\author[Veerapen]{Padmini Veerapen}
\address{(Veerapen) Department of Mathematics, Tennessee Tech University, Cookeville, TN 38505, U.S.A.}
\email{pveerapen@tntech.edu}

\author[Wang]{Xingting Wang}
\address{(Wang) Department of Mathematics, Howard University, Washington, DC 20059, U.S.A.}
\email{xingting.wang@howard.edu}

\date\today

\subjclass{
16T05, %Hopf Algebras and their applications 
16W50, %Graded Rings and modules (associative rings and algebras)
17B37  %Quantum groups (quantized enveloping algebras) and related deformations)
}
\keywords{cogroupoid, superpotential algebra, preregular form, universal quantum group, 2-cocycle twist, Artin--Schelter regular algebra}

\begin{document}
\begin{abstract} 
We construct a family of cogroupoids associated to preregular forms and recover the Morita--Takeuchi equivalence for Artin--Schelter regular algebras of dimension two, observed by Raedschelders and Van den Bergh. Moreover, we study the 2-cocycle twists of pivotal analogues of these cogroupoids, by developing a categorical description of preregularity in any tensor category that has a pivotal structure. 
\end{abstract}

\maketitle  

%%%%%%%%%%%%%%%%%%%%%%%%%%%%%%%%%%%%%%%%%
\section{Introduction}

This paper examines superpotentials associated to Artin--Schelter (AS) regular algebras and their universal quantum groups via the construction of certain bi-Galois objects using the language of cogroupoids. Superpotentials, or their duals, preregular forms, can be associated to any $N$-Koszul AS-regular algebra \cite{Dubois-Violette2005} and play an important role in noncommutative algebra, noncommutative algebraic geometry, and quantum groups, for example, via the classification of algebras \cite{BSW,MSKoszul,Mori-Ueyama2019}. Quantum groups associated to these objects were introduced independently by Dubois-Violette and Launer \cite{DVL1990} and by Wang  \cite{Wang1995}. Later, Bichon and Dubois-Violette \cite{BDV13} gave an explicit presentation of this quantum group by generators and relations. By \cite{Chirvasitu-Walton-Wang2019}, when the superpotential algebra is $N$-Koszul and AS-regular, this quantum group coincides with Manin's universal quantum group (c.f.~\cite{Manin2018}), the Hopf algebra that universally coacts on the underlying algebra. These quantum groups and their generalizations, which consists of a wide class of Hopf algebras, including the coordinate rings $\mc{O}(\GL_n)$ and their quantum analogues $\mc{O}_q(\GL_n)$ (as in \cite{BG2002}), have recently been studied in \cite{Chirvasitu-Walton-Wang2019, WaltonWang2016}.

Schauenburg showed in \cite{Sch1996} that the categories of comodules over two Hopf algebras $H$ and $L$ are monoidally equivalent, called Morita--Takeuchi equivalent, if and only if there exists an $H$-$L$-bi-Galois object between them. Later, Bichon \cite{Bichon2014} introduced the notion of a cogroupoid to provide a categorical context for Hopf-(bi)Galois objects. An understanding of the structure of cogroupoids is useful since it enhances the classical theory by allowing categorical arguments on Hopf-(bi)Galois objects. Other applications of cogroupoids include: explicit construction of new resolutions from old ones in homological algebra, invariant theory, monoidal equivalences between categories of Yetter--Drinfeld modules with applications to bialgebra cohomology and Brauer groups \cite{Bichon2013, Bichon2014}. Here, we construct a cogroupoid whose objects are determined by preregular forms. 

\begin{introdeprop}[\Cref{lem:4.7}, \Cref{lem:antipode}, \Cref{defn:he}]
For any integer $m \geq 2$, there is a cogroupoid $\mathcal{GL}_m$, whose objects are given by all $m$-preregular forms. In particular, for any $m$-preregular form $f$, the Hopf algebra $\mathcal{GL}_m(f,f)$ is the universal quantum group associated to $f$, as given in \cite{Chirvasitu-Walton-Wang2019}.
\end{introdeprop}

Since Dubois-Violette showed that any $N$-Koszul AS-regular algebra is a superpotential algebra associated to some preregular form \cite{Dubois-Violette2005}, our construction provides an explicit way to establish the Morita--Takeuchi equivalence between Manin's universal quantum groups associated to $N$-Koszul AS-regular algebras. In particular, we recover a special case of a result of Raedschelders and Van den Bergh from \cite{vdb2017} stating the Morita--Takeuchi equivalence between Manin's universal quantum groups associated to AS-regular algebras of the same dimension. Our method does not depend on the categorical approach of the Tannaka--Krein formalism, but relies instead on the non-vanishing of certain bi-Galois objects between any $N$-Koszul AS-regular algebras.

\begin{introdethm}[\Cref{prop:AS2}]
Manin's universal quantum groups associated to any two AS-regular algebras of dimension two are Morita--Takeuchi equivalent.
\end{introdethm}

Moreover, we study $\mathcal{SL}_m$-type universal quantum groups under $2$-cocycle twisting. As a necessary condition, we introduce the notion of a preregular morphism, a generalization of a preregular form, in any rigid tensor category that has a pivotal structure. Pivotal (also called sovereign) categories have been important in topological quantum field  theory. The pivotal structure allows the definition of quantum dimension, which can be used to produce numerical invariants of 3-manifolds and knots \cite{Hennings1996,KR1995}. When the (co)representation category of a Hopf algebra is pivotal, the Hopf algebra is called (co)pivotal (or (co)sovereign, as in \cite{Bichon2001}). In our paper, we employ the pivotal structure to define a $\Hom$-space operator $D_V^m$ which resembles the cyclic permutation of tensor products of vector spaces $V^{\otimes m}$. This enables us to define the notion of preregulariy on morphisms in a categorical context.

\begin{introdefn}[\Cref{defn:preop}]
Let $\mathcal C$ be a pivotal tensor category. For any integer $m\ge 2$ and $V\in \ob(\mc{C})$, a morphism $f: V^{\otimes m}\to \unit$ is called \emph{preregular} if 
\begin{itemize}
    \item[(1)] $f$ is \emph{non-degenerate}, namely there is a surjection $\pi: V^{\otimes (m-1)}\twoheadrightarrow \!^*V$ for the right dual object $\!^*V$ of $V$ in $\mc{C}$ such that the following diagram 
    \[
    \xymatrix{
    V^{\otimes m}\ar[rr]^-{f} \ar@{=}[d] && \unit\\
    V\otimes V^{\otimes (m-1)}\ar[rr]_-{\id_{V}\otimes \pi} && V\otimes\, \!^*V\ar[u]_-{\ev}
    }
    \]
     commutes, and 
    \item[(2)] $D_V^m(f)=f$, where the operator $D_V^m: {\rm Hom}_\mathcal C(V^{\otimes m}, \unit)\to {\rm Hom}_\mathcal C(V^{\otimes m}, \unit)$ is defined in \eqref{eq:DV}.
\end{itemize}
\end{introdefn}

We observe that the inverse duals of these $\Hom$-space operators $D_V^m$ are exactly those $E_{V^*}^m$, where $V^*$ denotes the dual of $V$, used to define generalized Frobenius--Schur indicators in an arbitrary pivotal category \cite{Ng-Schauenburg, NgSchauenburg2010}. As a consequence, any preregular morphism is an eigenvector of the above operator $D_V^m$ satisfying some nondegeneracy conditions. Moreover, the dual of a preregular morphism, which we refer to as a twisted superpotential, is invariant under the operator $E_{V^*}^m$. We use this generalization to construct another cogroupoid associated to preregular forms.

\begin{introdeprop}[\Cref{def:SL}, \Cref{lemma:univ-sl}]
For any integer $m \geq 2$, there is a cogroupoid $\mathcal{SL}_m$, whose objects are given by all $m$-preregular forms. In particular, for any $m$-preregular form $f$, the Hopf algebra 
\[\mathcal{SL}_m(f,f) = \mathcal{GL}_m(f,f)/(D-1) \] 
is the universal copivotal Hopf algebra associated to $f$, as given in \cite{Bichon2001,BDV13}. Here, $D$ is the quantum determinant of $\mathcal{GL}_m(f,f)$.
\end{introdeprop}

Using the cogroupoid $\mathcal{SL}_m$, we obtain the following result on the 2-cocycle twists of the universal quantum groups of preregular forms considered in \cite{BDV13}, which are copivotal Hopf algebras. 

\begin{introdethm}[\Cref{thm:SLE}]
Let $m \ge 2$ be an integer and $V$ be a finite-dimensional $\kk$-vector space. Let $f$ be an $m$-preregular form on $V$ and $\sigma$ be a left 2-cocycle on $\mathcal{SL}_m(f,f)$. Then the twisted map $f_\sigma$ (see \Cref{DefPreregularForm}) is also an $m$-preregular form on $V$ and the universal quantum groups
\[
\mathcal{SL}_m(f_\sigma,f_\sigma)~\cong~\mathcal{SL}_m(f,f)^\sigma\] 
are isomorphic as Hopf algebras. 
\end{introdethm}

\subsection*{Acknowledgements} 
The authors thank Chelsea Walton and Ken Goodearl for useful discussions. Some results in this paper were formulated at the Structured Quartet Research Ensembles (SQuaREs) program in March 2022, and at the BIRS Workshop on Noncommutative Geometry and Noncommutative Invariant Theory in September 2022. The authors thank the American Institute of Mathematics, the Banff International Research Station, and the organizers of the BIRS Workshop for their hospitality and support. Nguyen was partially supported by the Naval Academy Research Council and NSF grant DMS-2201146. Ure was partially supported by an AMS--Simons Travel Grant. Vashaw was partially supported by an Arthur K. Barton Superior Graduate Student Scholarship in Mathematics from Louisiana State University, NSF grants DMS-1901830 and DMS-2131243, and NSF Postdoctoral Fellowship DMS-2103272. Wang was partially supported by Simons Collaboration grant \#688403 and AFOSR grant FA9550-22-1-0272.

%%%%%%%%%%%%%%%%%%%%%%%%%%%%%%%%%
\section{Preliminaries}
\label{sec:superpotential}
Throughout the paper, let $\kk$ be a base field with $\otimes$ taken over $\kk$ unless stated otherwise. All categories are $\kk$-linear and all algebras are associative over $\kk$. We use the Sweedler's notation for the coproduct in a coalgebra $B$: $\Delta(h) = \sum h_1 \otimes h_2$ for any $h \in B$. When a Hopf algbera $H$ (right) coacts on an algebra $A$, we denote the coaction $\rho: A \to A \otimes H$ by $a \mapsto \sum a_0 \otimes a_1$. The category of all (resp.~ finite-dimensional) right $B$-comodules is denoted by ${\rm comod}(B)$ (resp.~ ${\rm comod}_{\rm fd}(B))$.

In this section, we present some background on superpotential algebras associated to preregular forms, cogroupoids, and 2-cocycle twists. In \cite[Theorem 4.3]{Dubois-Violette2005}, Dubois-Violette proved that every $N$-Koszul AS-regular algebra of finite global dimension $d$, generated by $n$ elements in degree one, is a twisted superpotential algebra. In this paper, AS-regular algebras will refer to Gorenstein algebras with finite global dimension; one may view them as ``nice" noncommutative analogues of polynomial rings. Note that we do not require AS-regular algebras to have finite Gelfand--Kirillov dimension.

%%%%%%%%%%%%%%%%%%%%%%%%%%%%%%%%%
\subsection{Superpotential algebras and Manin's universal quantum groups} 

We use the definitions from \cite{Dubois-Violette2005} of an algebra associated to a preregular form. 

\begin{defn}
\label{defn:preregular}
Let $2 \leq N \leq m$ be integers and $V$ be an $n$-dimensional $\kk$-vector space.
\begin{enumerate}
    \item An $m$-linear form $f$ on $V$ is called \emph{preregular} if it satisfies the following conditions:
    \begin{enumerate} 
        \item $f \left(v_1, v_2, \ldots, v_m \right) = 0$ for any $v_2, \ldots, v_m \in V$ implies that $v_1=0$, and
        \item there is some $\mathbb{P} \in \GL(V)$ so that 
        \[f\left( v_1, \ldots, v_m \right) = f\left( \mathbb{P} (v_m), v_1, \ldots, v_{m-1}\right), \text{ for all }v_1, \ldots, v_m \in V.\]
    \end{enumerate} 
    Given a preregular form $f$ on a vector space $V$ with fixed basis $\{v_1, \ldots, v_n\}$, we will typically denote by $f_{i_1 \cdots i_m} = f \left( v_{i_1}, \ldots, v_{i_m}\right)$, for any $1 \leq i_1,\ldots,i_m \leq n$.
    \item Let $f$ be an $m$-preregular form on $V$, and $\{v_1, \ldots, v_n\}$ be a fixed basis of $V$. The \emph{superpotential algebra associated to $f$}, denoted by $A(f,N)$, is the $\kk$-algebra generated by $n$ generators $x_1,\ldots,x_n$ subject to the relations 
    \[\sum_{1\leq j_1,\ldots,j_N\leq n} f_{i_1 \cdots i_{m-N} j_{1} \ldots j_N}\, x_{j_1}\cdots x_{j_N} =0 \]
    for every possible $1\leq i_1,\ldots,i_{m-N}\leq n$. 
\end{enumerate}    
\end{defn} 
For any $m$-preregular form $f$ on $V$, it is straightforward to check that there is an associated $\mathbb P\in \GL_n(\kk)=\GL(V)$ satisfying 
\begin{align}
\label{eq:preregular}
\sum_{1\leq i\leq n} \mathbb P_{ii_m}f_{ii_1\cdots i_{m-1}}=f_{i_1\cdots i_m} \qquad \text{and} \qquad \sum_{1\leq i\leq n} (\mathbb P^{-1})_{ii_1}f_{i_2\cdots i_{m}i}=f_{i_1\cdots i_m}
\end{align}
for every possible $1\leq i_1,\ldots,i_m\leq n$. Henceforth, an algebra $A$ will be called a \emph{superpotential algebra} if there are some choice of integers $m$ and $N$ with $2 \leq N \leq m$, and an $m$-preregular form $f$ so that $A \cong A(f,N)$. Here, any superpotential algebra can be considered as a graded algebra by assigning degree 1 to its generators.
   
\begin{remark}
\label{defn:superpotential}
Using the notation from \Cref{defn:preregular}, let $c: V^{\otimes m} \rightarrow V^{\otimes m}$ be the linear map defined by 
\[c\left( v_1 \otimes \cdots \otimes v_m \right) := v_m \otimes v_1 \otimes \cdots \otimes v_{m-1}, \quad \text{for any $v_i \in V$}.\]
An element $s \in V^{\otimes m}$ is a \emph{twisted superpotential} if there is $\mathbb{P} \in \GL(V)$ so that 
\[\left( \mathbb{P} \otimes \id^{\otimes(m-1)} \right) c(s) = s.\]
Given a twisted superpotential $s \in V^{\otimes m}$, the \emph{superpotential algebra} associated to $s$ is defined as
\[A(s,N) := TV / \partial^{m-N} (\Bbbk s),\]
where $TV$ is the tensor algebra on $V$,
\begin{align*} 
\partial  (\Bbbk s) &= \left\{ \left( \nu \otimes \id^{\otimes (m-1)}\right) (\alpha s) \mid \nu \in V^*, \alpha \in \Bbbk \right\}, 
\text{ and } \\
\partial^{i+1} (\Bbbk s) &= \partial \left( \partial^i ( \Bbbk s)\right), \quad \text{for all }i \geq 0. \end{align*} 
By identifying $V^{\otimes m} \cong \left( ( V^*)^{\otimes m} \right)^*$, where $(-)^*$ denotes the $\kk$-dual, there is a one-to-one correspondence between $m$-preregular forms on $V^*$ and twisted superpotentials in $V^{\otimes m}$: 
 \[ {{\left\{ 
 \text{$m$-preregular forms on } V^* \right\} 
\overset{1:1}{\longleftrightarrow} \left\{ \text{twisted superpotentials in } V^{\otimes m} \right\}}}. \]
Furthermore, the associated algebras $A(f,N)$ and $A(s,N)$ are isomorphic for $f$ associated to $s$ under the above correspondence \cite[Lemma 2.4]{Chirvasitu-Walton-Wang2019}. 
\end{remark} 

Next, we review Manin's construction of the universal quantum group $\underline{\rm aut}(A)$ associated to any superpotential algebra $A=A(f,N)$ as described in \cite{Manin2018}. Note that the original definition was only given under the assumption that $A$ is a quadratic algebra, but it can be generalized to any graded algebra. 

\begin{defn}
Let $A$ be a $\mathbb Z$-graded algebra. 
\begin{enumerate}
    \item We say a Hopf algebra $H$ \emph{left coacts on $A$ preserving the grading of $A$} via $\rho: A\to  H\otimes A$ if each homogeneous component of $A$ is a left $H$-comodule via $\rho$ and $\rho$ is an algebra map. In this case, we say $A$ is a left graded comodule algebra over $H$.
    \item \emph{Manin's left universal quantum group $\underline{\rm aut}^l(A)$ associated to $A$} is the Hopf algebra that left coacts on $A$ preserving the grading of $A$ via $\rho: A\to \underline{\rm aut}^l(A)\otimes A$ satisfying the following universal property: If $H$ is any Hopf algebra that left coacts on $A$ preserving the grading of $A$ via $\tau: A\to H\otimes A$, then there is a unique Hopf algebra map $f: \underline{\rm aut}^l (A)\to H$ such that the diagram 
\begin{align}\label{def:aut}
\xymatrix{
A\ar[r]^-{\rho}\ar[dr]_-{\tau} &  \underline{\rm aut}^l(A)\otimes A\ar[d]^-{f\otimes \id} \\
& H\otimes A
}
\end{align}
commutes. Similarly, we can define $\underline{\rm aut}^r(A)$ by using the universal right coaction on $A$ preserving the grading of $A$.
\end{enumerate}
\end{defn}

By \cite[Example 4.8(1)-(2)]{AGV}, we know $\underline{\rm aut}^l(A)$ always exists if $A$ is locally finite, namely, when $\dim_\kk A_i<\infty$ for all $i\in \mathbb Z$. In particular, when $A=A(f,N)$ is a superpotential algebra, $\underline{\rm aut}^l(A)$ and $\underline{\rm aut}^r(A)$ always exist.

\begin{Ex}\cite[\S 3.2]{RVdB2017}\label{E:GL2}
Manin's left universal quantum group $\underline{{\rm aut}}^l(A)$ associated to the polynomial algebra $A=\kk[x,y]$ is generated by the entries of the $2 \times 2$ matrix 
\[M=\begin{pmatrix} a& b\\ c&d \\ \end{pmatrix}\]
together with the formal inverse of the determinant $\delta=ad-cb$, subject to the following relations:
\begin{align*}
ac -ca~=\,&0=~bd -db,\\
a\delta^{-1}d-b\delta^{-1}c ~=\,&1=~ d\delta^{-1}a- c\delta^{-1}b, \text{ and } \\ b\delta^{-1}a-a\delta^{-1}b~=\,&0 =~ c\delta^{-1}d-d\delta^{-1}c.
\end{align*}
The bialgebra structure of $\underline{{\rm aut}}^l(A)$ is given by $\Delta(M)=M\otimes M$ and $\varepsilon(M)=\begin{pmatrix} 1& 0\\ 0&1 \\ \end{pmatrix}$. The antipode is determined by $S(M)=\begin{pmatrix} \delta^{-1}a& -\delta^{-1}b\\ -\delta^{-1}c&\delta^{-1}d\end{pmatrix}$. As consequence, $\delta$ is a group-like element in $\underline{{\rm aut}}^l(A)$. 
\end{Ex}

%%%%%%%%%%%%%%%%%%%%%%%%%%%%%%%%
\subsection{Cogroupoids}

We now discuss Morita--Takeuchi equivalence in the context of the universal quantum groups associated to preregular forms using the language of cogroupoids introduced by Bichon \cite{Bichon2014}. These provide a categorical framework for bi-Galois objects discussed by Schauenburg \cite{Sch1996}.
\begin{defn}
\label{defn:cocategory}
A \emph{$\kk$-cocategory} $\mc{C}$ consists of:
\begin{enumerate}
\item A set of objects $\ob(\mc{C})$;
\item For any $X,Y\in \ob(\mc{C})$, a $\kk$-algebra $\mc{C}(X,Y)$;
\item For any $X,Y,Z\in \ob(\mc{C})$, $\kk$-algebra homomorphisms
\[ \bdt^Z_{XY}:\mc{C}(X,Y)\ra \mc{C}(X,Z)\ot \mc{C}(Z,Y) \qquad \t{and} \qquad \vps_X:\mc{C}(X,X)\ra \kk \]
such that for any $X,Y,Z,T\in \ob(\mc{C})$, the diagrams
\[
\xymatrix{
\mc{C}(X,Y)\ar^-{\bdt^Z_{X,Y}}[rr]\ar^{\bdt^T_{X,Y}}[d] && \mc{C}(X,Z)\ot \mc{C}(Z,Y)\ar^-{\bdt^T_{X,Z}}[d]\\
\mc{C}(X,T)\ot \mc{C}(T,Y)\ar^-{\id\ot \bdt^Z_{T,Y}}[rr]&&\mc{C}(X,T)\ot \mc{C}(T,Z)\ot \mc{C}(Z,Y), }
\]
\[\xymatrix
{\mc{C}(X,Y)\ar@{=}[rd]\ar^{\bdt^Y_{X,Y}}[d]&\\
\mc{C}(X,Y)\ot\mc{C}(Y,Y)\ar^-{\id \ot \vps_Y}[r]&\mc{C}(X,Y),} \qquad 
\xymatrix
{\mc{C}(X,Y)\ar@{=}[rd]\ar^{\bdt^X_{X,Y}}[d]&\\
\mc{C}(X,X)\ot\mc{C}(X,Y)\ar^-{\vps_X \ot \id }[r]&\mc{C}(X,Y)}
\]
commute.
\end{enumerate}
\end{defn}

For $a^{X,Y}\in \mc{C}(X,Y)$, we use Sweedler's notation to write 
\[ \bdt^Z_{X,Y}(a^{X,Y})=\sum a^{X,Z}_1\ot a^{Z,Y}_2. \]
From its definition, a cocategory with one object is just a bialgebra. In particular, $\mc{C}(X, X)$ is a bialgebra for any $X \in \ob(\mc{C})$. A cocategory $\mc{C}$ is said to be \emph{connected} if $\mc{C}(X,Y)$ is a nonzero algebra for any $X,Y\in \ob(\mc{C})$.

\begin{defn}
\label{defn:cogroupoid}
A \emph{$\kk$-cogroupoid} $\mc{C}$ consists of a $\kk$-cocategory $\mc{C}$ together with linear maps
\[ S_{X,Y}:\mc{C}(X,Y)\longrightarrow \mc{C}(Y,X), \]
for any $X,Y\in \ob(\mc{C})$, such that the diagram 
{\small \[\xymatrix{\mc{C}(Y,X) & \kk \ar[l]_-u & \mc{C}(X,X)\ar[d]_{\bdt_{X,X}^Y}\ar[r]^-{\vps_X}\ar[l]_-{\varepsilon_X} &\kk\ar[r]^-u&\mc{C}(X,Y)\\
\mc{C}(Y,X)\ot\mc{C}(Y,X) \ar[u]^m &&\mc{C}(X,Y)\ot\mc{C}(Y,X)\ar[rr]^{\id\ot S_{Y,X}} \ar[ll]_-{S_{X,Y} \otimes \id} &&\mc{C}(X,Y)\ot\mc{C}(X,Y) \ar[u]^m} \]}
commutes. 
\end{defn}

The following proposition describes properties of the ``antipodes" in cogroupoids. For other properties of cogroupoids, we refer the reader to \cite{Bichon2014}. In a cogroupoid $\mc{C}$, the bialgebra $\mc{C}(X,X)$ is a Hopf algebra for any $X\in \ob(\mc{C})$, with the antipode map $S_{X,X}$ described here. 
\begin{proposition}\cite[Proposition 2.13]{Bichon2014}
Let $\mc{C}$ be a cogroupoid and $X,Y\in\ob(\mc{C})$. Then the following hold. 
\begin{enumerate}
\item[(1)] $S_{Y,X}:\mc{C}(Y,X)\ra \mc{C}(X,Y)^{\operatorname{op}}$ is an algebra homomorphism.
\item[(2)] For any $Z\in \ob(\mc{C})$ and $a^{Y,X}\in\mc{C}(Y,X)$,
\[ \bdt_{X,Y}^Z(S_{Y,X}(a^{Y,X}))=\sum S_{Z,X}(a_2^{Z,X})\ot S_{Y,Z}(a_1^{Y,Z}). \]
\end{enumerate}
\end{proposition}

The following is Bichon's reformulation of Schauenburg's result \cite{Sch1996} about bi-Galois objects for Morita--Takeuchi equivalences in terms of cogroupoids. 
\begin{thm}\cite[Theorem 2.10]{Bichon2014}
Let $H$ and $L$ be Hopf algebras. The following assertions are equivalent.
 \begin{itemize}
     \item[(1)] There exists a $\kk$-linear equivalence of monoidal categories ${\rm comod}(H)\stackrel{\otimes }{\cong} {\rm comod}(L)$;
     \item[(2)] There exists a connected cogroupoid $\mathcal C$ and two objects $X,Y \in \ob(\mc{C})$ such that $H=\mathcal C(X, X)$ and $L=\mathcal C(Y, Y )$.
 \end{itemize}
\end{thm}

%%%%%%%%%%%%%%%%%%%%%%%%%%%%%%%%%
\subsection{Pivotal tensor categories}
\label{subsec:pivotal} 

We now recall some concepts and notation from the theory of tensor categories. We use $(\mathcal C,\,\otimes,\,\Phi,\, \mathbb{1},\,r,\,\ell)$ to denote a $\kk$-linear tensor category, with a bifunctor $\otimes: \mathcal C\times \mathcal C\to \mathcal C$ called the \emph{tensor product}, a natural isomorphism $\Phi_{X,Y,Z}: (X\otimes Y)\otimes Z\xrightarrow{\sim} X\otimes(Y\otimes Z)$ called the \emph{associativity constraint}, a \emph{unit object} $\mathbb 1$ together with natural isomorphisms $r_X: X\otimes \mathbb 1\xrightarrow{\sim} X$ and $\ell_X: \mathbb 1\otimes X \xrightarrow{\sim} X$ called the \emph{right and left unit constraints} for all $X,Y,Z \in \ob(\mc{C})$. These structure maps are subject to the Pentagon and Triangle axioms \cite[XI.2.1]{KasGTM}. If two objects $X,Y\in \ob(\mc{C})$ are obtained by tensoring together the same sequence of objects with two different arrangements of parentheses, one can then construct a natural isomorphism between them by composing several instances of the tensor products of $\Phi,\Phi^{-1}$ and the identity. Any above isomorphism is unique by Mac Lane's coherence theorem \cite{MacLane63}, and will be denoted by $\Phi^{?}: X\to Y$. By \cite{Schau01}, we may always assume that the unit object $\mathbb{1}$ of $\mathcal C$ is strict, namely $X\otimes \mathbb{1}=X=\mathbb{1}\otimes X$ and the left and right unit constraints $\ell_{\scriptscriptstyle X}$ and $r_{\scriptscriptstyle X}$ are just the identity maps for every $X \in \ob(\mc{C})$. 

Let $\mathcal C$ and $\mathcal D$ be two tensor categories. Any \emph{tensor functor} $(\mathcal F,\xi): \mathcal C\to \mathcal D$ consists of a functor $\mathcal F:\mathcal C\to \mathcal D$ and a natural isomorphism $\xi_{\scriptscriptstyle X,Y}: \mathcal F(X)\otimes \mathcal F(Y)\xrightarrow{\sim} \mathcal F(X\otimes Y)$ for any $X,Y\in \ob(\mc{C})$ satisfying the monoidal structure axiom (see e.g., \cite[Definition 2.4.1]{EGNO} and \cite[XI.4.1]{KasGTM}). Along with the strictness of $\mathbb 1$, we also assume that $\mathcal F(\mathbb {1})=\mathbb{1}$ with $\xi_{\mathbb 1,X}=\xi_{X,\mathbb 1}={\rm id}_{\mathcal F(X)}$ \cite[Remark 2.4.6]{EGNO}. 

A \emph{left dual} of an object $V\in \ob(\mc{C})$ is an object $V^*$ together with two morphisms $\ev: V^*\otimes V\to \mathbb{1}$ and $\coev: \mathbb{1}\to V\otimes V^*$ such that 
\[
\id_V=(V \xrightarrow{\coev \otimes \id_V} (V \otimes V^*) \otimes V\xrightarrow{\Phi} V \otimes (V^* \otimes V) \xrightarrow{\id_V \otimes \ev} V)
\]
and
\[
\id_{V^*} = (V^* \xrightarrow{\id_{V^*} \otimes \coev} V^* \otimes (V \otimes V^*) \xrightarrow{\Phi} (V^* \otimes V) \otimes V^* \xrightarrow{\ev \otimes \id_{V^*}} V^*).
\]
We say $\mathcal C$ is \emph{left rigid} if every object of $\mathcal C$ admits a left dual. A right dual of an object and right rigidity can be defined similarly for $\mathcal C$. Suppose $\mathcal C$ is left rigid. Then $(-)^*$ is a contravariant monoidal functor together with a monoidal structure $\zeta: Y^*\otimes X^*\xrightarrow{\sim} (X\otimes Y)^*$. As a consequence, the double left dual $(-)^{**}: \mathcal C\to \mathcal C$ is a monoidal functor. 

\begin{defn}\label{defn:pivotal-cat}
A $\kk$-linear left rigid tensor category $(\mathcal C,\,\otimes,\,\Phi,\, \mathbb{1},\,r,\,\ell)$ is called \emph{pivotal} if there is a natural isomorphism $j: {\rm id}_\mathcal C\to (-)^{**}$ of monoidal functors. In this case, $j$ is called a \emph{pivotal structure of $\mathcal{C}$}.
\end{defn}

If $\mathcal C$ and $\mathcal D$ are two pivotal categories, and $(\mathcal F,\xi): \mathcal C\to \mathcal D$ is a monoidal functor, we say $\mathcal F$ \emph{preserves the pivotal structure}, if the diagram
\begin{equation}\label{equ:piv}
\begin{tikzcd}[column sep=normal, row sep=normal]
\mathcal F(V)\arrow[r,"\mathcal F(j)"]\arrow[d,"j_{\mathcal F(V)}"']&\mathcal F(V^{**})\arrow[d,"\widetilde{\xi}"]\\
\mathcal F(V)^{**}\arrow[r,"\widetilde{\xi}^*"] & \mathcal F(V^*)^*
\end{tikzcd}
\end{equation}
commutes, where the natural isomorphism $\widetilde{\xi}_V: \mathcal F(V^*)\to \mathcal F(V)^*$, called the \emph{duality transformation}, is uniquely determined by $(\mathcal F,\xi)$ (see \cite[\S 1]{ Ng-Schauenburg}). A \emph{strict pivotal category} is a strict monoidal category with a pivotal structure in which both the monoidal functor structure $\xi$ of $(-)^*$ and the pivotal structure $j$ are identities. By \cite[Theorem 2.2]{Ng-Schauenburg}, every pivotal category is equivalent, as a pivotal category, to a strict one.

%%%%%%%%%%%%%%%%%%%%%%%%%%%%%%%%%
\subsection{2-cocycle twists} 

Schauenburg proved in \cite{Sch1996} that Morita--Takeuchi equivalences for Hopf algebras are in bijection with bi-Galois objects. A subset of these equivalences correspond to 2-cocycle twists, which were introduced by Doi and Takeuchi \cite{Doi93,DT94} (in fact, when the Hopf algebra is finite-dimensional, every Morita--Takeuchi equivalence arises from a 2-cocycle). In this section, we give a brief overview of 2-cocycle twists, and in \Cref{sec:4} we discuss 2-cocycle twists of preregular forms and their associated universal pivotal cogroupoids.

\begin{defn} Let $H$ be a Hopf algebra over a field $\kk$.  
A \emph{2-cocycle} on $H$ is a convolution invertible linear map $\sigma: H\otimes H\to \kk$ satisfying 
\begin{equation}\label{eq:2cocycle}
\begin{aligned}
   \sum \sigma(x_1, y_1)\,\sigma (x_2y_2, z)&=\sum \sigma(y_1, z_1)\,\sigma(x, y_2z_2), \quad \text{and}\\
   \sigma(x,1)&=\sigma(1,x)=\varepsilon(x),
\end{aligned}
\end{equation} 
for all $x,y,z\in H$. The convolution inverse of $\sigma$ is usually denoted by $\sigma^{-1}$. 
\end{defn}

Given a 2-cocycle $\sigma: H\otimes H\to\kk$, let $H^\sigma$ denote the coalgebra $H$ endowed with the original unit and deformed product
\[
x*_\sigma y:=\sum \sigma(x_1,y_1)\,x_2y_2\,\sigma^{-1}(x_3,y_3)
\]
for any $x,y\in H$. In fact, $H^\sigma$ is a Hopf algebra with the deformed antipode $S^\sigma$ given in \cite[Theorem 1.6]{Doi93}. We call $H^\sigma$ the \emph{2-cocycle twist} of $H$ by $\sigma$. It is well-known that two Hopf algebras are 2-cocycle twists of each other if and only if there exists a bicleft object between them (e.g., see \cite{Sch1996}). 

Now, suppose $\sigma: H \otimes H\to \kk$ is a left 2-cocycle on $H$. It is well-known that there is a monoidal equivalence between the category of comodules of $H$ and that of $H^\sigma$ given by
\begin{align}
\label{eq:equiv2}
  (F,\xi) : \mathrm{comod}(H) \overset{\sim}{\rightarrow} \mathrm{comod}(H^\sigma),  
\end{align}
where $F$ is identity as functor on objects together with the monoidal functor structure 
\begin{align}\label{def: MTF}
\xi_{U,V}: F(U\otimes V)&\longrightarrow F(U)\otimes_\sigma F(V)\\
u\otimes v&\longmapsto \sum \sigma(u_1,v_1)\,u_0\otimes v_0, \notag
\end{align}
for any $u\in U$ and $v\in V$, with inverse $\xi_{U,V}^{-1}: u\otimes v\mapsto \sum \sigma^{-1}(u_1,v_1)u_0\otimes v_0$. Moreover, the set of all left 2-cocycles on a Hopf algebra $H$ gives rise to an associated 2-cocycle cogroupoid $\underline{H}$ as follows.

\begin{Ex}\cite[Definition 3.14]{Bichon2014}
\label{2-coc-cog}
Let $H$ be any Hopf algebra. The {\it 2-cocycle cogroupoid} of $H$, denoted by $\underline{H}$, is defined as follows:
\begin{itemize}
    \item[(1)] ${\rm ob}($\underline{H}$)=Z^2(H)$, which is the set of all left 2-cocycles on $H$.
    \item[(2)] For any $\sigma,\tau\in Z^2(H)$, the algebra $H(\sigma,\tau)$ is defined such that $H(\sigma,\tau)=H$ as vector spaces with the new multiplication
\[
x * y=\sum \sigma(x_1,y_1)\, x_2y_2\tau^{-1}(x_3,y_3),
\]
for any $x,y\in H(\sigma,\tau)$.
\item[(3)] The structural maps $\Delta^{\bullet}_{\bullet,\bullet}$, $\varepsilon_\bullet$ and $S_{\bullet,\bullet}$ are given as: for any $\sigma,\tau,\omega\in Z^2(H)$,
\begin{align*}
    \Delta^\omega_{\sigma,\tau}=\Delta: H(\sigma,\tau)&\longrightarrow H(\sigma,\omega)\otimes H(\omega,\tau)\\
    x&\longmapsto \sum x_1\otimes x_2, \\
\varepsilon_\sigma=\varepsilon: H(\sigma,\sigma)&\longrightarrow \kk, \\
S_{\sigma,\tau}: H(\sigma,\tau)&\longrightarrow H(\tau,\sigma)\\
x&\longmapsto \sum \sigma(x_1,S(x_2))\, S(x_3)\, \tau^{-1}(S(x_4),x_5)).
\end{align*}
\end{itemize}
\end{Ex}

\begin{Ex}
As in \Cref{2-coc-cog}, let $H$ be any Hopf algebra. Take $\underline{H}_Z$ to be the full subcogroupoid of $\underline{H}$, where the objects correspond to those 2-cocycles arising from twisting pairs (using the language of \cite{HNUVVW21}). In this case, for a 2-cocyle $\sigma$ on $H$, the Hopf algebra $\underline{H}_Z(\sigma,\sigma)$ is equal to a twist (in the sense of \cite{ATV1991, Zhang1996}) of $H$ by a graded automorphism (see \cite[Remark 2.9]{Bichon-Neshveyev-Yamashita2016} or \cite[Theorem E]{HNUVVW21}).
\end{Ex}

%%%%%%%%%%%%%%%%%%%%%%%%%%%%%%%%%
\section{The cogroupoids associated to preregular forms}

In this section, we introduce a family of cogroupoids associated to preregular forms and discuss their properties. In \Cref{prop:AS2}, using the language of cogroupoids, we show the Morita--Takeuchi equivalence for AS-regular algebras of dimension two.

%%%%%%%%%%%%%%%%%%%%%%%%%%%%%%%%%
\subsection{Construction of the cogroupoid $\mathcal {GL}_m$}

Let $V$ be a $k$-dimensional vector space and $W$ be an $l$-dimensional vector space over $\kk$ with fixed bases $\{v_1,...,v_k\}$ of $V$ and $\{w_1,..., w_l\}$ of $W$. For any integer $m \ge 2$, let $e$ be an $m$-linear preregular form on $V$ and $f$ be an $m$-linear preregular form on $W$. Recall that we write  $e\left(v_{i_1}, \dots, v_{i_m} \right)= {e}_{i_1 \cdots i_m}\in \kk$ and $f\left(w_{i_1}, \dots, w_{i_m} \right)= {f}_{i_1 \cdots i_m}\in \kk$, respectively.

\begin{defn}\label{def:GL}
We define $\mathcal{GL}_m(e,f)$ to be the $\kk$-algebra with $(2kl+2)$ generators 
    \[\mathbb A=(a_{ij})_{\substack{1\leq i\leq k \\ 1\leq j\leq l}}, \qquad \mathbb B=(b_{ij})_{\substack{1\leq i\leq l \\ 1\leq j\leq k}}, \qquad D^{\pm 1},\]
subject to the relations
\begin{equation}
\label{eq:alg}
\left. 
\begin{aligned}
    \sum_{1\leq i_1,\ldots,i_m\leq k}{e}_{i_1 \cdots i_m}a_{i_1j_1}\cdots a_{i_mj_m} &= {f}_{j_1\cdots j_m}D, &\textnormal{ for any } 1\leq j_1,\dots, j_m\leq l,\\
    \sum_{1\leq i_1,\ldots,i_m\leq l}{f}_{i_1\cdots i_m}b_{i_mj_m}\cdots b_{i_1j_1} &= {e}_{j_1\cdots j_m}D^{-1},  &\textnormal{ for any } 1\leq j_1,\dots, j_m\leq k,\\
    DD^{-1} &= D^{-1}D=1, \text{ and } \\
    \mathbb{A}\mathbb{B} &= \mathbb{I}_{k\times k}.
\end{aligned}\right\}
\end{equation}
We denote the generators of  $\mathcal{GL}_m(e,f)$ by $a_{ij}^{e,f}, b_{ij}^{e,f},$ and $\left( D^{e,f}\right)^{\pm 1}$ when multiple preregular forms are involved, and omit the superscripts when the context is clear. In particular, if $W=V$ and $f=e$, then we simply write $\mathcal{GL}_m (e)=\mathcal{GL}_m(e,e)$. 
\end{defn}

\begin{remark}
\label{gl-manin}
We note that $\mathcal{GL}_m(e)$ is the algebra $\mathcal H(e)$ associated to a preregular form $e$, as defined in \cite[Definition 5.1]{Chirvasitu-Walton-Wang2019}. It is a Hopf algebra with the Hopf structure given in \cite[Proposition 5.8]{Chirvasitu-Walton-Wang2019}. When the superpotential algebra $A(e,N)$ is $N$-Koszul and Gorenstein, $\mathcal{GL}_m(e)$ is Manin's universal quantum group coacting on $A(e,N)$ (see \cite[Theorem 5.33]{Chirvasitu-Walton-Wang2019}). 
\end{remark}

\begin{lemma}
\label{lem:4.7}
For any integer $m\geq 2,$
$\mathcal{GL}_m$ forms a $\kk$-cocategory, where the objects are $m$-linear preregular forms on $\kk$-vector spaces. In particular, for any vector spaces $U,V,W$ with $\dim U=p,$ $\dim V=q,$ and $\dim W=r$, for any $m$-linear preregular forms $e$ on $U$, $f$ on $V$ and $g$ on $W$, there exist algebra maps
\[ \Delta = \Delta_{e,g}^f: \mathcal{GL}_m(e,g)\to \mathcal{GL}_m(e,f)\otimes \mathcal{GL}_m(f,g) \]
such that 
\begin{align*}
\Delta(a_{ij}^{e,g})&=\sum_{k=1}^q a^{e,f}_{ik}\otimes a^{f,g}_{kj}, \qquad \text{for } 1\leq i\leq p, 1\leq j\leq r,\\
\Delta(b_{ji}^{e,g})&=\sum_{k=1}^q b^{e,f}_{ki}\otimes b^{f,g}_{jk}, \qquad \text{for } 1\leq i\leq p, 1\leq j\leq r,\\
\Delta((D^{e,g})^{\pm 1})&=(D^{e,f})^{\pm 1}\otimes (D^{f,g})^{\pm 1}, \end{align*} and
\[ \varepsilon_{e}: \mathcal{GL}_m(e)\to \kk\]
such that $\varepsilon_{e}(a_{ij}^{e,e})=\varepsilon_{e}(b_{ji}^{e,e})=\delta_{ij}$, for $1\leq i,j\leq p$, and $\varepsilon_e((D^{e,e})^{\pm 1})=1$. 
\end{lemma}

\begin{proof}
Since $\Delta = \Delta_{e,g}^f$ is already defined on the generators of $\mathcal{GL}_m(e,g)$, it suffices to verify that it preserves all relations:
\[
\begin{aligned}
     &\Delta\left(\sum_{1\leq i_1,\ldots,i_m\leq p}{e}_{i_1 \cdots i_m}a_{i_1j_1}\cdots a_{i_mj_m}\right)\\
     &= \sum_{\substack{1\leq i_1,\ldots,i_m\leq p\\1\leq k_1,\ldots,k_m\leq q}} \left(e_{i_1\cdots i_m} a_{i_1k_1}\cdots a_{i_mk_m}\right)\otimes a_{k_1j_1}\cdots a_{k_mj_m}\\
     &= \sum_{1\leq k_1,\ldots,k_m\leq q} f_{k_1\cdots k_m} D\otimes a_{k_1j_1}\cdots a_{k_mj_m} \\
     &=\sum_{1\leq k_1,\ldots,k_m\leq q}   D\otimes \left(f_{k_1\cdots k_m} a_{k_1j_1}\cdots a_{k_mj_m}\right)\\
     &= D\otimes g_{j_1\cdots j_m} D = \Delta\left(g_{j_1\cdots j_m}D\right).
     \end{aligned}
     \]
Similarly, we may also show that 
\[ 
\begin{aligned}
    \Delta\left(\sum_{1\leq i_1,\ldots,i_m\leq r}{g}_{i_1\cdots i_m}b_{i_mj_m}\cdots b_{i_1j_1}\right)
     =&\Delta\left(e_{j_1\ldots j_m}D^{-1}\right).
\end{aligned}
\]
On the other hand, one can check that 
\[\Delta\left(DD^{-1}\right)=\Delta\left(D\right)\Delta\left(D^{-1}\right)=1 \otimes 1=\Delta\left(1\right)=\Delta\left(D^{-1}\right)\Delta\left(D\right)=\Delta\left(D^{-1}D\right), \]
and 
\[
\begin{aligned}
\Delta\left(\sum_{1\leq k\leq r} a_{ik}b_{kj}\right) &=\sum_{1\leq k\leq r} \sum_{1\leq s,t\leq q}a_{is}b_{tj}\otimes a_{sk}b_{kt}= \sum_{1\leq s,t\leq q}a_{is}b_{tj}\otimes \left(\sum_{1\leq k\leq r}a_{sk}b_{kt}\right)\\
&=\sum_{1\leq s,t\leq q}\delta_{st}a_{is}b_{tj}\otimes 1 
=\sum_{1\leq s\leq q} a_{is}b_{sj}\otimes 1=\delta_{ij} 1 \otimes 1=\Delta(\delta_{ij} 1).
\end{aligned}
\]
Hence $\Delta_{e,g}^f: \mathcal{GL}_m(e,g)\to \mathcal{GL}_m(e,f)\otimes \mathcal{GL}_m(f,g)$ is a well-defined algebra map. Note that $\varepsilon: \mathcal{GL}_m(e)\to \kk$ is a well-defined algebra map since $\mathcal{GL}_m(e)$ is a Hopf algebra \cite[Proposition 5.8]{Chirvasitu-Walton-Wang2019}. 
It remains to show that the diagrams in \Cref{defn:cocategory} commute, which is straightforward on the generators. 
\end{proof}

The following result is similar to \cite[Lemma 5.6]{Chirvasitu-Walton-Wang2019} in the context of cogroupoids, and we leave its proof to the reader. 

\begin{lemma}
\label{lem:4.9}
Let $V$ and $W$ be $\kk$-vector spaces of dimension $k$ and $l$, respectively, Consider two $m$-preregular forms $e$ and $f$ on $V$ and $W$, respectively, together with invertible matrices $\mathbb P\in \GL_k(\kk) = \GL(V)$ and $\mathbb Q\in \GL_l(\kk) = \GL(W)$ as in \eqref{eq:preregular}. Then the following equalities hold in $\mathcal{GL}_m(e,f)$:
\begin{gather*}
    D^{-1}\mathbb Q^T\mathbb A^T\mathbb P^{-T} D\mathbb B^T=\mathbb B\mathbb A=\mathbb I_{l \times l} \qquad \text{and} \qquad \mathbb B^TD^{-1}\mathbb Q^T\mathbb A^T\mathbb P^{-T}D=\mathbb A\mathbb B=\mathbb I_{k \times k}.
\end{gather*}
\end{lemma}

\begin{lemma}
\label{lem:antipode}
For any integer $m \geq 2$, the cocategory $\mc{GL}_m$ forms a cogroupoid, with $\Delta$ and $\varepsilon$ defined as in \Cref{lem:4.7}, and the algebra map
\[
S_{e,f}: \mathcal{GL}_m(e,f)\to \mathcal{GL}_m(f,e)^{\operatorname{op}}
\]
is defined by the formulas 
\begin{align*}
S_{e,f}\left(\mathbb A^{e,f}\right)&=\mathbb B^{f,e},\\
S_{e,f}\left(\mathbb B^{e,f}\right)&=\left(D^{f,e}\right)^{-1}\,\mathbb Q^{-1}\,\mathbb A^{f,e}\,\mathbb P\, D^{f,e},\\ S_{e,f}\left(\left(D^{e,f}\right)^{\pm 1}\right)&=\left(D^{f,e}\right)^{\mp 1}.
\end{align*}
\end{lemma}

\begin{proof}
Set $\dim V=k$ and $\dim W=l$. We first show that $S_{e,f}: \mathcal {GL}_m(e,f)\to \mathcal{GL}_m(f,e)^{\operatorname{op}}$ preserves the relations in $\mathcal{GL}_m(e,f)$ and hence it is a well-defined algebra map. One can see that
\[
\begin{aligned}
  S_{e,f}\left(\sum_{1\leq i_1,\ldots,i_m\leq k}{e}_{i_1 \cdots i_m}a_{i_1j_1}\cdots a_{i_mj_m}\right)&=\sum_{1\leq i_1,\ldots,i_m\leq k}{e}_{i_1 \cdots i_m}S_{e,f}\left(a_{i_mj_m}\right)\cdots S_{e,f}\left(a_{i_1j_1}\right)\\
  &=\sum_{1\leq i_1,\ldots,i_m\leq k}{e}_{i_1 \cdots i_m}b_{i_mj_m}\cdots b_{i_1j_1}\\
  &=f_{j_1\cdots j_m}D^{-1}\\
  &=S_{e,f}\left(f_{j_1\ldots j_m}D\right),
\end{aligned}
\]
and similarly,
\[
\begin{aligned}
  S_{e,f}\left(\sum_{1\leq i_1,\ldots,i_m\leq l}{f}_{i_1 \cdots i_m}a_{i_mj_m}\cdots a_{i_1j_1}\right)
  &=S_{e,f}\left(e_{j_1\ldots j_m}D^{-1}\right).
\end{aligned}
\]
Moreover, it is clear that $S_{e,f}$ preserves the relation $DD^{-1}=D^{-1}D=1$ and 
\[
\begin{aligned}
 S_{e,f}\left(\sum_{1\leq p\leq l} a_{ip}b_{pj}\right)&=\sum_{1\leq p\leq l} S(b_{pj})S\left(a_{ip}\right)=\sum_{1\leq p\leq l} (D^{-1}\mathbb Q^{-1}\mathbb A\mathbb PD)_{pj}\mathbb B_{ip}\\
 &=\sum_{1\leq p\leq l} (D^{-1}P^T\mathbb A^T\mathbb Q^{-T}D)_{jp}\mathbb B^T_{pi}=(D^{-1}\mathbb P^T\mathbb A^T\mathbb Q^{-T}D\mathbb B^T)_{ij}=\delta_{ij}.
\end{aligned}
\]
The last equality holds in $\mathcal{GL}_m(f,e)$ by \Cref{lem:4.9}. 
It remains to show the commutativity of the two diagrams in \Cref{defn:cogroupoid}, which is straightforward to check on generators. 
\end{proof}

We summarize the discussion above in the following definition/theorem, which follows similarly to \cite[Definitions 2.1 and 2.4]{Bichon2014}.

\begin{defn}
\label{defn:he}
For any integer $m\ge 2$, the cogroupoid $\mathcal {GL}_m$ is defined as follows:
\begin{enumerate}
    \item ${\rm ob}(\mathcal{GL}_m)=\{e: V^{\otimes m}\to \kk\mid \text{$e$ is a preregular form on some finite-dimensional vector space $V$}\}$.
    \item For $e,f\in {\rm ob}(\mathcal{GL}_m)$, $\mathcal {GL}_m(e,f)$ is the algebra defined in \Cref{def:GL}.
    \item The structural maps $\Delta_{\bullet,\bullet}^\bullet,\ \varepsilon_\bullet, \ S_{\bullet,\bullet}$ are given in \Cref{lem:4.7,lem:antipode}.
\end{enumerate} 
\end{defn}

%%%%%%%%%%%%%%%%%%%%%%%%%%%%%%%%%
\subsection{Connectivity of $\mathcal {GL}_2$}

It is proved in \cite[Theorem 7.2.3]{vdb2017} that the universal quantum groups of any two Koszul AS-regular algebras are Morita--Takeuchi equivalent as long as the two algebras share the same global dimension. We present an alternate criterion for the Morita--Takeuchi equivalence of universal quantum groups, using the language of cogroupoids.

\begin{proposition}
\label{thm:FormMorita}
Let $2 \leq N, N'\leq m$ be three integers, and $V,W$ be two finite-dimensional $\kk$-vector spaces. Let $e$ and $f$ be two $m$-preregular forms on $V$ and $W$, respectively, such that the associated superpotential algebras $A=A(e,N)$ and $B=A(f,N')$ are two $N$ and $N'$-Koszul AS-regular algebras. If the algebra $\mathcal{GL}_m(e,f)\neq 0$, then the universal quantum groups $\underline{\rm aut}^l(A)$ and $\underline{\rm aut}^l(B)$ are Morita--Takeuchi equivalent.
\end{proposition}

\begin{proof}
By \cite[Theorem 5.33]{Chirvasitu-Walton-Wang2019}, we have $\underline{\rm aut}^l(A(e,N)) \cong \mathcal{GL}_m(e)$ and $\underline{\rm aut}^l(A(f,N')) \cong \mathcal{GL}_m(f)$. The result follows from \cite[Theorems 2.10, 2.12]{Bichon2014} and the cogroupoid construction in \Cref{defn:he}.
\end{proof}

In the following, we give an application of \Cref{thm:FormMorita} for the case when $m=2$ and $e:V^{\otimes 2}\to \kk$ and $f: W^{\otimes 2}\to \kk$ are two preregular forms. As a consequence, it provides another proof of a special case of \cite[Theorem 7.2.3]{vdb2017} for AS-regular algebras of dimension two. Recall that in the classification of AS-regular algebras of dimension two \cite[Theorem 0.1]{Zhang1998}, being AS-regular is equivalent to being a superpotential algebra.

\begin{thm}
\label{prop:AS2}
Let $A$ and $B$ be any two AS-regular algebras of dimension two. Then $\uaut^l(A)$ and $\uaut^l(B)$ are Morita--Takeuchi equivalent. 
\end{thm}

\begin{proof}
By the classification of AS-regular algebras of dimension 2 in \cite[Theorem 0.1]{Zhang1998}, $A$ and $B$ can be presented as superpotential algebras $A=A(e,2)$ and $B=A(f,2)$, respectively, for some $2$-preregular forms $e$ and $f$. Hence, by Remark \ref{gl-manin} and \Cref{thm:FormMorita}, it suffices to show that the bi-Galois object  $\mc{GL}_2(e,f)$ between $\uaut^l(A) = \mc{GL}_2(e) $ and $\uaut^l(B) = \mc{GL}_2(f)$ is nonzero. 

Suppose $e$ and $f$ are preregular forms on vector spaces $V$ and $W$, respectively, of dimensions $k$ and $l$. We fix a basis $\{v_1,\ldots,v_k\}$ for $V$ and  write $e$ as a matrix $\mathbb E\in M_k(\kk)$ such that $E_{ij}=e(v_i,v_j)$ for $1 \leq i,j \leq k$. It is easy to check that $e$ is a preregular form if and only if $\mathbb E\in \GL(V) = \GL_k(\kk)$; in this case, the twisting matrix for $e$ is given by $\mathbb P=\mathbb E^{-T}\mathbb E$. Similarly, we denote the matrix $\mathbb F\in \GL(W)=\GL_l(\kk)$ associated to the preregular form $f$, and the twisting matrix for $f$ is given by $\mathbb Q=\mathbb F^{-T}\mathbb F$.

By \eqref{eq:alg}, the $\kk$-algebra $\mathcal{GL}_2(e,f)$ 
is presented by $(2kl+2)$ generators 
    \[\mathbb A=(a_{ij})_{\substack{1\leq i\leq k \\ 1\leq j\leq l}}, \qquad \mathbb B=(b_{ij})_{\substack{1\leq i\leq l \\ 1\leq j\leq k}}, \qquad D^{\pm 1},\]
    subject to the relations
 \begin{align}\label{them:relAut}
    \mathbb A^T\,\mathbb E\,\mathbb A =\mathbb F\,D, \quad \mathbb B^T\,\mathbb F^T\,\mathbb B=\mathbb E^T\,D^{-1},\quad  \mathbb A\, \mathbb B=\mathbb I_{k \times k}, \text{ and } \quad DD^{-1}=D^{-1}D=1.  
  \end{align}
By \Cref{lem:4.9}, we also have $\mathbb B\,\mathbb A=\mathbb I_{l \times l}$. Hence  
\begin{align*}
    \mathbb E^T\,D^{-1}\mathbb A=(\mathbb B^T\,\mathbb F^T\,\mathbb B)\mathbb A=\mathbb B^T\,\mathbb F^T\,(\mathbb B\,\mathbb A)=\mathbb B^T\,\mathbb F^T.
\end{align*}
This implies that $\mathbb B^T=D^{-1}\,\mathbb E^T\,\mathbb A\,\mathbb F^{-T}$, and so $\mathbb B=D^{-1}\,\mathbb F^{-1}\,\mathbb A^T\,\mathbb E$. Hence 
$\mc{GL}_2(e,f)$ is the quotient of the free algebra $\kk\langle \mathbb A,D^{\pm 1}\rangle$ by the relations
\begin{align*}
\mathbb A^T\,\mathbb E\,\mathbb A=\mathbb F\, D,\quad \mathbb A\, D^{-1}\,\mathbb F^{-1}\,\mathbb A^{T}=\mathbb E^{-1}, \text{ and } \quad DD^{-1}=D^{-1}D=1.
\end{align*}
Recall the algebra $\mathscr B(\mathbb E,\mathbb F)$ defined in \cite[Definition 3.1]{Bichon2003}. One checks directly that 
\[\mc{GL}_2(e,f)/(D-1)\cong \mathscr B(\mathbb E,\mathbb F).\]
For arbitrary $\mathbb F \in \GL(W)$ with corresponding preregular form $f$, let $k=2$ and 
\[\mathbb E=\mathbb E_q:=\begin{pmatrix} 
0 & 1\\ -q^{-1}  &0
\end{pmatrix}\in \GL_2(\kk)
\]
such that $q^2+{\rm tr}(\mathbb F^T\mathbb F^{-1})+1=0$. Denote $e_q$ as the preregular form corresponding to $\mathbb{E}_q$. Then $
\mc{GL}_2(e_q,f)\neq 0$ by \cite[Proposition 3.4]{Bichon2003}. 

Now in view of \cite[Proposition 2.15]{Bichon2014}, to show that $\mc{GL}_2(e,f) \neq 0$ it suffices to show that $\mc{GL}_2(e_q,e_{q'}) \neq 0$ for $q,q' \in \kk^\times$, which follows from the lemma below. 
\end{proof}

\begin{lemma}
Let $e$ and $f$ be two $2$-preregular forms on the same vector space $V$. Then $\mc{GL}_2(e,f) \neq 0.$ 
\end{lemma}

\begin{proof}
Use the notation from the previous proof.
Note that from the presentation given above, by setting the degree of each $a_{ij}$ to be 1, the degree of each $b_{ij}$ to be $-1$, and the degree of $D^{\pm 1}$ to be $\pm 2$, $\mathcal{GL}_2(e,f)$ is a $\mathbb{Z}$-graded algebra. Let $\mathbb M \in$ $\GL(V)=\GL_k(\kk)$. We construct a nonzero graded representation $ U$ for $\mathcal {GL}_2(e,f)$ based on the matrix $\mathbb M$ in the following way:
\begin{enumerate}
    \item As a vector space, set $U:=\bigoplus_{d \in \mathbb{Z}} U_d$, where each $U_d$ is defined to be the 1-dimensional vector space $\kk$.
    \item Set $\mathbb M_0 := \mathbb M$ and inductively define
    \[
    \begin{cases} 
    \mathbb M_{d+1} := \mathbb{E}^{-T} \mathbb M_d^{-T} \mathbb F^T, & \text{for } d \geq 0\\
    \mathbb{M}_{d-1}:=\mathbb E^{-1} \mathbb M_d^{-T} \mathbb F, & \text{for } d \leq 0.
    \end{cases}
    \]
    \item Define the action of $\mathbb{A}$ on each graded component $U_d$ to be given by scalar multiplication $U_d \to U_{d+1}$, according to the matrix $\mathbb{M}_d$. Similarly, define the action of $\mathbb{B}$ on the graded component $U_d$ to be given by scalar multiplication $U_d \to U_{d-1}$, given by the matrix $\mathbb M_{d-1}^{-1}.$ This gives the following diagram, where the action of $\mathbb{A}$ moves to the right, and the action of $\mathbb{B}$ moves to the left:
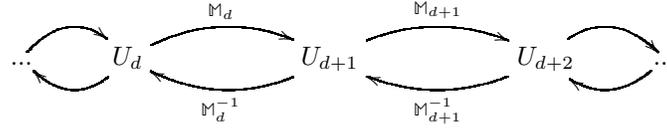
\[
\xymatrix{
... \ar@/^1.0pc/[r]^-{}&U_d \ar@/^1.0pc/[l]^-{} \ar@/^1.0pc/[rr]^-{\mathbb{M}_{d}} && U_{d+1} \ar@/^1.0pc/[rr]^-{\mathbb{M}_{d+1}} \ar@/^1.0pc/[ll]^-{\mathbb{M}_{d}^{-1}} && U_{d+2} \ar@/^1.0pc/[ll]^-{\mathbb{M}_{d+1}^{-1}}\ar@/^1.0pc/[r]^-{}&...\ar@/^1.0pc/[l]^-{}
}
\]
\item The action of $D^{\pm 1}$ on $U_d$ will be defined as the multiplication by 1 from $U_d \to U_{d \pm 2}$. 
\end{enumerate}

By the equalities
\begin{align*}
    \mathbb{M}^T_{d+1} \mathbb E \mathbb M_d = \mathbb F, \qquad
    \mathbb M_{d-1}^{-T} \mathbb F^T \mathbb M_d^{-1} = \mathbb E^T, \text{ and } \qquad
    \mathbb{M}_d \mathbb{M}_d^{-1} = \mathbb I_{k \times k}, 
\end{align*}
for all $d \in \mathbb{Z}$, these actions respect the relations \eqref{them:relAut}. Hence $U$ is a nonzero representation for $\mathcal{GL}_2(e,f)$, and so this algebra is itself nonzero. 
\end{proof}

The following result is a straightforward consequence of the proof of \Cref{prop:AS2}.

\begin{Cor}
\label{connected}
The cogroupoid $\mathcal{GL}_2$ is connected.
\end{Cor}

\begin{Ex}
\label{ex-grobner}
Let $\mathbb E$ and $\mathbb F$ be identity matrices on $k$ and $l$-dimensional vector spaces $V$ and $W$, respectively, with $l$ and $k \geq 2$. In this case, it can be checked that the following collection of relations forms a noncommutative Gr{\"o}bner basis (see e.g., \cite[Section 1.4]{Ro2016}) for the ideal of relations for $\mathcal{GL}_2(e,f)$, under the graded lexicographic ordering with the variables ordered $a_{11}>a_{12}>...>a_{21}>... >b_{11}>b_{12}>...>b_{lk} > D>D^{-1}$ (and where $\delta_{hi}$ in all formulas represents the Kronecker delta function):
\begin{align*}
     \mathbb A \mathbb B &= \mathbb I_{k \times k},&
     \mathbb B \mathbb A &= \mathbb I_{l \times l},&
     D^{-1} \mathbb A &= \mathbb B^T, \\
     D \mathbb B &= \mathbb A^T,&
     DD^{-1}&=1,&
     D^{-1} D &= 1, \\
     \mathbb B^T \mathbb B &= D^{-1} \mathbb I_{k \times k},&
     \mathbb A^T \mathbb A &= D \mathbb I_{l \times l};
     \end{align*}
     for all triples $1 \leq h,i,j \leq k,$ the relations 
     \begin{align*}
         \sum_{l \geq m > 1} ( a_{hm}b_{mi}b_{1j} - a_{h1} b_{mi} b_{mj}) &=   \delta_{hi} b_{1j}-\delta_{ij} a_{h1} D^{-1},\\
         \sum_{l \geq m > 1} ( b_{mh}b_{mi}b_{1j} - b_{1h} b_{mi} b_{mj}) &=  \delta_{hi} D^{-1} b_{1j}-\delta_{ij} b_{1h} D^{-1};
     \end{align*}
     and for all triples $1 \leq h,i,j \leq l,$ the relations 
     \begin{align*}
     \sum_{k \geq m >1} (b_{hm} a_{mi} a_{1j} - b_{h1} a_{mi} a_{mj}) &= \delta_{hi} a_{1j}-\delta_{ij} b_{h1} D,\\
     \sum_{k \geq m >1} (a_{mh} a_{mi} a_{1j} - a_{1h} a_{mi} a_{mj}) &=  \delta_{hi} D a_{1j}-\delta_{ij} a_{1h} D.
     \end{align*}
     It follows that there exists a basis for the algebra $\mathcal{GL}_2(e,f)$ consisting of all monomials in the variables $\mathbb A, \mathbb B, D$, and $D^{-1}$ which do not contain any leading term in the above list of relations. In particular, we can observe concretely that this algebra is nonzero. This calculation follows similarly to the argument given by Cohn in \cite[Section 5]{Cohn1966} for the related algebra with generators $\mathbb A$ and $\mathbb B$ and relations $\mathbb A \mathbb B = \mathbb I_{k \times k}$, $\mathbb B \mathbb A = \mathbb I_{l \times l}$, which was originally constructed by Leavitt \cite{Leavitt1957}. Note, however, that the additional relations for $\mathcal{GL}_2(e,f)$ lead to a more complicated normal form. 
\end{Ex}

%%%%%%%%%%%%%%%%%%%%%%%%%%%%%%%%%%%%%%%%%
\section{Preregular forms in pivotal tensor category and their 2-cocycle twists}
\label{sec:4}

In this section, we introduce categorical descriptions of superpotentials and preregular forms in pivotal tensor categories. We also study their realization in comodule categories over copivotal Hopf algebras. We refer the reader to \Cref{subsec:pivotal} for some background on such categories. 

%%%%%%%%%%%%%%%%%%%%%%%%%%%%%%%%%%%%%%%%%
\subsection{Twisted superpotentials and preregular forms in pivotal tensor categories}
Consider a $\kk$-linear Hom-finite tensor category $(\mathcal C,\,\otimes,\,\Phi,\, \mathbb{1},\,r,\,\ell)$ with pivotal structure $j: {\rm id}_\mathcal C\to (-)^{**}$. For any $V\in \ob(\mc{C})$, we denote by $V^{\otimes n}$ the $n$-fold tensor product of $V$ with rightmost parentheses. This means $V^{\otimes 0}=\mathbb{1}$, and $V^{\otimes (n+1)}=V\otimes V^{\otimes n}$. Following the notation in \cite{Ng-Schauenburg}, there is a unique isomorphism:
\[
\Phi^{(n)}: V^{\otimes (n-1)}\otimes V\to V^{\otimes n}
\]
that is inductively defined by $\Phi^{(1)}: V \to V$ is the identity, and for $n \geq 1$
\[
\Phi^{(n+1)}=\left(V^{\otimes n}\otimes V=(V\otimes V^{\otimes (n-1)})\otimes V\xrightarrow{\Phi} V\otimes (V^{\otimes (n-1)}\otimes V)\xrightarrow{\id_V \otimes \Phi^{(n)}} V\otimes V^{\otimes n}=V^{\otimes (n+1)}\right).
\]
We now generalize the notion of a prereregular form to its categorical analogue.
\begin{defn}
\label{defn:preop}
Let $\mathcal C$ be a tensor category with pivotal structure $j$. For any integer $m\ge 2$ and object $V$ in $\mathcal C$, we define the following.
\begin{itemize}
    \item[(1)] A morphism $f: V^{\otimes m}\to \unit$ is called \emph{non-degenerate} if there is a surjection $\pi: V^{\otimes (m-1)}\twoheadrightarrow \!^*V$ for the right dual \;$\!^*V$ such that the diagram 
    \[
    \xymatrix{
    V^{\otimes m}\ar[rr]^-{f}\ar@{=}[d] && \unit\\
    V\otimes V^{\otimes (m-1)}\ar[rr]_-{\id_V\otimes \pi} && V\otimes\, \!^*V\ar[u]_-{\ev}
    }
    \]
     commutes.
    \item[(2)] A non-degenerate morphism $f: V^{\otimes m}\to \unit$ is called \emph{$m$-preregular of characteristic $q$} if $D_V^m(f)= qf$, for some $q \in \kk^\times$. Here, the operator $D_V^m: {\rm Hom}_\mathcal C(V^{\otimes m}, \unit)\to {\rm Hom}_\mathcal C(V^{\otimes m}, \unit)$ is defined as
\begin{align}\label{eq:DV}
D_V^m(f)&~:=~\left(V^{\otimes m}\xrightarrow{{\coev}\otimes \id} (\!^*V\otimes V)\otimes V^{\otimes m}\xrightarrow{\id \otimes (\Phi^{(m)})^{-1}} (\!^*V\otimes V)\otimes (V^{\otimes (m-1)}\otimes V)\right.\notag\\
&\quad \left.\xrightarrow{\Phi^?} \!^*V\otimes (V^{\otimes m}\otimes V)\xrightarrow{\id \otimes (f\otimes \id)} \!^* V\otimes V\xrightarrow{j_{\,\! {^*V}}\otimes \id} V^*\otimes V\xrightarrow{\ev} \unit\right),
\end{align}
for any $f\in {\rm Hom}_\mathcal C(V^{\otimes m}, \unit)$. See \Cref{subsec:pivotal} for the definition of $\Phi^?$.
\end{itemize}
We simply say $f: V^{\otimes m}\to \unit$ is \emph{preregular} if $f$ is preregular of characteristic $q=1$.
\end{defn}

\begin{Ex}
  Let $\mathcal{C}$ be the tensor category $\mathrm{Vec}_\kk$ of all finite-dimensional $\kk$-vector spaces. It is clear that $\mathrm{Vec}_\kk$ is pivotal with the pivotal structure given by the natural identification $V\cong V^{**}$ for any finite-dimensional vector space $V$ over $\kk$. Remark the operator $D_V^m$ on ${\rm Hom}_\kk(V^{\otimes m},\kk)$ defined in \Cref{defn:preregular}(2) is given by $D_V^m(f)(v_1,\ldots,v_m)=f(v_m,v_1,\ldots,v_{m-1})$ for any $\kk$-linear map $f: V^{\otimes m}\to \kk$ and $v_1,\ldots, v_m\in V$. Moreover, $f$ is an $m$-preregular form of character $q$ on $V$ if and only if it is an $m$-preregular form in the sense of \Cref{defn:preregular} with the matrix $\mathbb P\in \GL(V)$ given by $\mathbb P={\rm diag}(q,\ldots, q)$.    
\end{Ex}

To connect the generalization of a preregular form from \Cref{defn:preop} to the notion of superpotential, we recall some Hom-space operators used in the definition of higher Frobenius--Schur indicators, which play the same role as the cyclic condition in \Cref{defn:preregular}(1)(b) of a preregular form. Remark that superpotentials in the categorical context are slightly different from the formal dual of preregular forms. 

\begin{defn}\cite[Definition 3.1]{Ng-Schauenburg}
Let $\mathcal C$ be a tensor category with pivotal structure $j$ and $V,W \in \ob(\mathcal C)$.
\begin{itemize}
    \item[(1)] Define $T_{V,W}: \Hom_\mathcal C(V^*,W)\to \Hom_\mathcal C(W^*,V)$ via $T_{V,W}(f)=\left(W^*\xrightarrow{f^*} V^{**}\xrightarrow{{j_V}^{-1}} V\right)$. 
    \item[(2)] Define $A_{V,W}: \Hom_\mathcal C(\mathbb{1}, V\otimes W)\to \Hom_\mathcal C(V^*,W)$ via 
    \[A_{V,W}(h)~=~\left(V^*\xrightarrow{\id_{V^*} \otimes h} V^*\otimes (V\otimes W)\xrightarrow{\Phi^{-1}} (V^*\otimes V)\otimes W\xrightarrow{\ev\otimes \id_W} W\right)\]
    with inverse 
    \[
A^{-1}_{V,W}(g)~=~\left(\mathbb{1}\xrightarrow{\coev} V\otimes V^*\xrightarrow{\id_V \otimes g} V\otimes W\right).
    \]
    \item[(3)] For any $n\ge 1$, define 
    \begin{align*}
E_V^{(n)}~=~\left(\Hom_\mathcal C(\mathbb{1}, V^{\otimes n})\xrightarrow{A_{V,V^{\otimes (n-1)}}}\Hom_\mathcal C(V^*, V^{\otimes (n-1)})\xrightarrow{T_{V,V^{\otimes (n-1)}}} \Hom_\mathcal C((V^{\otimes (n-1)})^*,V)\right.\\
\left. \xrightarrow{A_{V^{\otimes (n-1)},V}^{-1}} \Hom_\mathcal C(\mathbb{1}, V^{\otimes (n-1)}\otimes V)\xrightarrow{\Hom_\mathcal C(\mathbb{1}, \Phi^{(n)})} \Hom_\mathcal C(\mathbb{1},V^{\otimes n}) \right).
    \end{align*}
\end{itemize}  
\end{defn}

We now define a superpotential in a categorical context by making use of the above Hom-space operators.
\begin{defn}\label{defn:superpotential2}
Let $\mathcal C$ be a tensor category with pivotal structure $j$. Let $m$ be an integer $\ge 2$ and $V \in \ob(\mathcal C)$. 
\begin{itemize}
\item[(1)] A morphism $s: \unit\to V^{\otimes m}$ is called \emph{non-degenerate} if the right dual $\,\!^*V$ is a subobject of $V^{\otimes (m-1)}$ with embedding  $\iota: \!^*V \hookrightarrow V^{\otimes (m-1)}$ such that the diagram 
    \[
    \xymatrix{
  \unit \ar[rr]^-{s}\ar[d]_-{\coev} && V^{\otimes m}\\
\!^*V\otimes V\ar[rr]_-{\iota \otimes \id_V} && V^{\otimes (m-1)}\otimes V\ar[u]_-{\Phi^{(m)}}
    }
    \]
    commutes. 
    \item[(2)] A non-degenerate morphism $s: \unit\to V^{\otimes m}$ is called a \emph{twisted-superpotential} if $E_V^{(m)}(s)=qs$, for some $q \in \kk^\times$.  In particular, $s$ is said to be a \emph{superpotential} if $s$ is non-degenerate and $q=1$. 
\end{itemize} 
\end{defn}

Remark that the above definition extends our earlier notion of superpotential in the category of finite-dimensional vector sapces ${\rm Vec}_\kk$: Let $V$ be an $n$-dimensional $\kk$-vector space and let $s: \kk\to V^{\otimes m}$ be a superpotential as in \Cref{defn:superpotential}. This is a superpotential in the sense of \Cref{defn:superpotential2}.

\begin{lemma}
\label{lem:conjugate}
Let $\mathcal C$ and $\mathcal D$ be any two pivotal tensor categories. For any pivotal equivalence $\mathcal F: \mathcal C\to \mathcal D$,  we have the following commutative diagram:
\[
\xymatrix{
{\rm Hom}_\mathcal C(V^{\otimes m},\,\unit)\ar[rr]^-{D_V^m}\ar[d]_-{\mathcal F} && {\rm Hom}_\mathcal C(V^{\otimes m},\,\unit)\ar[d]^-{\mathcal F}\\
{\rm Hom}_\mathcal D(\mathcal F(V^{\otimes m}),\,\unit)\ar[d]_-{{\rm Hom}_\mathcal D(\xi^{-1}_{V^{\otimes m}},\,\unit)} && {\rm Hom}_\mathcal D(\mathcal (V^{\otimes m}),\,\unit)\ar[d]^-{{\rm Hom}_\mathcal D(\xi^{-1}_{V^{\otimes m}},\,\unit)}
\\
{\rm Hom}_\mathcal D(\mathcal F\mathcal (V)^{\otimes m},\,\unit)\ar[rr]^-{D^m_{\mathcal F(V)}} && {\rm Hom}_\mathcal D(\mathcal F(V)^{\otimes m},\,\unit).
}
\]
In particular, pivotal equivalence preserves preregular forms of the same characteristic.  
\end{lemma}
\begin{proof}
Note that the operator $D_V^m$ can be interpreted as follows: for any $V,W \in \ob(\mathcal C)$, we define $D_{V,W}: \Hom_\mathcal C(V\otimes W,\unit)\to \Hom_\mathcal C(W\otimes V,\unit)$ such that
\begin{align*}
D_{V,W}(f)&:=\left(W\otimes V\xrightarrow{\coev\otimes \id} (\!^*V\otimes V)\otimes (W\otimes V)\xrightarrow{\Phi^? } \!^*V\otimes (V\otimes W)\otimes V\right.\\
&\quad \left.\xrightarrow{\id \otimes (f\otimes \id)} \!^*V\otimes V\xrightarrow{j_{\,\!{^*V}}\otimes \id} V^*\otimes V\xrightarrow{\ev} \unit\right),
\end{align*}
for any $f\in \Hom_\mathcal C(V\otimes W,\unit)$. We recall the operator $E_{V,W}: \Hom_\mathcal C(\unit, V\otimes W)\to \Hom_\mathcal C(\unit,W\otimes V)$ given in \cite[Definition 3.1]{Ng-Schauenburg} such that
\begin{align*}
E_{V,W}(f)&:=\left(\unit\xrightarrow{\coev\otimes \coev} (W\otimes W^*)\otimes (V^*\otimes V^{**})\xrightarrow{\Phi^? } (W\otimes (W^*\otimes V^*))\otimes V^{**}\right.\\
&\quad \xrightarrow{\id \otimes (f\otimes \id)} (W\otimes (W^*\otimes V^*))\otimes (V\otimes W)\otimes V^{**}\\
&\quad \xrightarrow{\Phi^?} W\otimes ((W^*\otimes V^*)\otimes (V\otimes W))\otimes V^{**}\\
&\quad \left.\xrightarrow{\id\otimes \ev\otimes \id} W\otimes V^{**}\xrightarrow{\id \otimes j_V^{-1}} W\otimes V \right).
\end{align*}
A straightforward computation shows that the diagram 
\begin{align}
\label{eq:DE}
    \xymatrix{
    \Hom_\mathcal C(V\otimes W,\unit)\ar[rr]^{D_{V,W}}\ar[d]_-{\cong} && \Hom_\mathcal C(W\otimes V,\unit)\ar[d]^-{\cong }\\
    \Hom_\mathcal C(\unit, (V\otimes W)^*)\ar[d]_-{\cong } && \Hom_\mathcal C(\unit, (W\otimes V)^*)\ar[d]^-{\cong }\\
    \Hom_\mathcal C(\unit, W^*\otimes V^*)\ar[rr]^-{E_{V^*,W^*}^{-1}}&& \Hom_\mathcal C(\unit, V^*\otimes W^*)
    }
\end{align}
commutes. Moreover, one sees that $D_V^m=D_{V,W}$ and $E_V^{(m)}=E_{V,W}$ (see \cite[Definition 3.1]{Ng-Schauenburg}) with $W=V^{\otimes (m-1)}$. So our result follows from \cite[Proposition 4.3]{Ng-Schauenburg} and the commutativity of \eqref{eq:DE}.
\end{proof}

By \Cref{lem:conjugate}, for the remainder of the paper we assume that the categories we consider are strict.

%%%%%%%%%%%%%%%%%%%%%%%%%%%%%%%
\subsection{Preregular forms in comodule categories over copivotal Hopf algebras}

In this subsection, we discuss preregular forms in an arbitrary comodule category that admits a pivotal structure \cite{FY, Y} and 2-cocycle twists of such preregular forms.

Let $(H,\Phi)$ be a copivotal Hopf algebra with character $\Phi$. It follows that the antipode $S$ of $H$ is bijective with inverse $S^{-1}=\Phi*S*\Phi^{-1}$. Moreover, by \cite[Proposition 3.10]{Bichon2001}, the category ${\rm comod}_{\rm fd}(H)$ admits a pivotal structure $\varphi: \!^*(-)\xrightarrow{\sim} (-)^*$. Here $\varphi$ is an isomorphism of monoidal functors between the right and left duality functors. For any finite-dimensional right $H$-comodule $V$, the linear map $\varphi_{V}: \!^*V\to V^*$ is defined as follows: $\varphi_V=(\id \otimes \Phi^{-1})\circ \rho_{\, \!^*V}$ (we use the fact that $\!^*V=V^*$ as vector spaces). To describe $\varphi_V$ explicitly, let $\{v_1,\ldots,v_n\}$ be a basis of $V$ with $\rho_V(v_i)=\sum_{i=1}^n v_j\otimes h_{ji}$ for $h_{ji}\in H$. Then 
\begin{align}\label{eq:cosovereign}
\varphi_V(v^i)=\sum_{j=1}^n \Phi(h_{ij})\,v^j
\end{align}
with dual basis $\{v^1,\ldots,v^n\}$ of $\!^*V$ and $V^*$.
We now generalize the notion of a copivotal (or cosovereign) Hopf algebra to a cogroupoid. 
 
\begin{defn}
A cogroupoid $\mathcal C$ is called \emph{copivotal}  
if for any $X\in {\rm ob}(\mathcal C)$, there exists some  
character $\Phi_X: \mathcal C(X,X)\to \kk$ such that for any $X,Y\in {\rm ob}(\mathcal C)$, 
\[
S_{Y,X}\circ S_{X,Y}=\Phi_X^{-1}*\id*\Phi_Y,
\]
where the right hand is defined as 
\begin{align*}
\Phi_X^{-1}*\id*\Phi_Y&:=\left(\mathcal C(X,Y)\xrightarrow{\Delta_{X,Y}^X} \mathcal C(X,X)\otimes \mathcal C(X,Y)\right.\\
&\quad \left.\xrightarrow{\id \otimes \Delta_{X,Y}^Y} \mathcal C(X,X)\otimes\mathcal C(X,Y)\otimes  C(Y,Y)\xrightarrow{\Phi_X^{-1}\otimes \id \otimes \Phi_Y} \mathcal C(X,Y)\right),
\end{align*}
and $\Phi_X^{-1} = \Phi_X \circ S_{X,X}$ is the convolution inverse of $\Phi_X$.
\end{defn}

Note that if the cogroupoid $\mathcal C$ only has one object $X$, then $\mathcal C$ is copivotal if and only if the Hopf algebra $\mathcal C(X,X)$ is a copivotal Hopf algebra (c.f.~\cite[Definition 3.7 and Remark 3.8]{Bichon2001}).

\begin{defn}\label{def:SL}
For any integer $m\ge 2$, we define the cogroupoid $\mathcal{SL}_m$ as
\begin{itemize}
    \item[(1)] ${\rm ob}(\mathcal {SL}_m)={\rm ob}(\mathcal{GL}_m)$.
    \item[(2)] For any two $m$-preregular forms $e$ and $f$, $\mathcal{SL}_m(e,f)=\mathcal{GL}_m(e,f)/(D-1)$ as in \eqref{eq:alg}.
    \item[(3)] The structural maps $\Delta_{\bullet,\bullet}^\bullet,\ \varepsilon_\bullet, \ S_{\bullet,\bullet}$ are all induced from those of $\mathcal{GL}_m$ in \Cref{defn:he}.
\end{itemize} 
We simply write ${\mathcal{SL}}_m(e)={\mathcal{SL}}_m(e,e)$. 
\end{defn}

\begin{Ex}
   Note that when $m=2$, $\mathcal{SL}_2$  coincides with the cogroupoid $\mathscr B(E,F)$ considered in \cite{Bichon2003}. While $\mathcal{GL}_2$ is connected by Corollary \ref{connected}, a straightforward calculation shows that for $e$ and $f$ as in Example \ref{ex-grobner}, where $l=2$ and $k=3$, we have $\mathcal{SL}_2(e,f)=0,$ and so $\mathcal{SL}_2$ is not connected.
\end{Ex}
 Let $V$ be a finite-dimensional $\kk$-vector space with a fixed basis $\{v_1,\ldots,v_n\}$ and $e: V^{\otimes m} \to \kk$  be an $m$-preregular form with associated matrix $\mathbb P\in \GL_n(\kk)$ subject to \eqref{eq:preregular}.
We point out that our $\mathcal {SL}_m(e)$ is the universal Hopf algebra $\mathcal H(e)$ described in \cite[\S 5]{BDV13}. Recall that $\mathcal{SL}_m(e)$ is generated over $\kk$ by $\mathbb A=(a_{ij})_{1\leq i,j\leq n}$ and $\mathbb B=(b_{ij})_{1\leq i,j\leq n}$ subject to relations:
\begin{equation}
\label{eq:SLalg}
\left. \begin{aligned}
    \sum_{1\leq i_1,\ldots,i_m\leq n}{e}_{i_1 \cdots i_m}a_{i_1j_1}\cdots a_{i_mj_m} &= {e}_{j_1\cdots j_m}, &\textnormal{ for any } 1\leq j_1,\dots, j_m\leq n,\\
    \sum_{1\leq i_1,\ldots,i_m\leq n}{e}_{i_1\cdots i_m}b_{i_mj_m}\cdots b_{i_1j_1} &= {e}_{j_1\cdots j_m},  &\textnormal{ for any } 1\leq j_1,\dots, j_m\leq n,\\
    \mathbb{A}\mathbb{B} &= \mathbb{I}_{n\times n},
\end{aligned}\right\}
\end{equation}
with Hopf algebra structure
\begin{gather*}
   \Delta(a_{ij})=\sum_{1\leq k\leq n} a_{ik}\otimes a_{kj},\quad  \Delta(b_{ji})=\sum_{1\leq k\leq n} b_{ki}\otimes b_{jk},\notag\\
   \varepsilon(a_{ij})=\delta(b_{ji})=\delta_{ij}\quad \text{for}\ 1\leq i,j\leq n,\\
   S(\mathbb A)=\mathbb B,\quad S(\mathbb B)=\mathbb P^{-1}\mathbb A\mathbb P\notag.
\end{gather*}
In particular, we know that $\mathcal{SL}_m(e)$ is copivotal with character $\Phi_e: \mathcal{SL}_m(e)\to \kk$ such that $\Phi_e(\mathbb A)=\mathbb P$ and $\Phi_e(B)=\mathbb P^{-1}$. We view $V$ as a right $\mathcal{SL}_m(e)$-comodule via $\rho: V\to V\otimes \mathcal{SL}_m(e)$ with $\rho(v_i)=\sum_{1\leq j\leq n} v_j\otimes a_{ji}$. Thus the $\mathcal{SL}_m(e)$-comodule map $e: V^{\otimes m}\to \kk$ has the following universal property:

\begin{lemma}
\label{lemma:univ-sl}\cite[Theorem 5]{BDV13}
Let $H$ be any Hopf algebra that right coacts on $V$ via $\rho_H: V\to V\otimes H$ such that $e: V^{\otimes m}\to \kk$ is $H$-colinear.  Then there is a unique Hopf algebra map $\theta: \mathcal{SL}_m(e)\to H$ such that the following diagram commutes: 
\[
\xymatrix{
V\ar[dr]_-{\rho_H}\ar[r]^-{\rho}& V\otimes {\mathcal{SL}}_m(e)\ar[d]^-{\id \otimes \theta}\\
& V\otimes H.
}
\]
\end{lemma}

\begin{lemma}
The cogroupoid $\mathcal{SL}_m$ is copivotal, with characters $\Phi_e: \mathcal{SL}_m(e)\to \kk$ given by $\Phi_e(\mathbb A)=\mathbb P$ and $\Phi_e(\mathbb B)=\mathbb P^{-1}$, that is, 
\[
S_{f,e}\circ S_{e,f}=\Phi_e^{-1}* \id *\Phi_f
\]
in $\mathcal{SL}_m(e,f)$, for any two $m$-preregular forms $e: V^{\otimes m}\to \kk$ and $f: W^{\otimes m}\to \kk$ with associated invertible matrices $\mathbb{P}$ and $\mathbb{Q}$, respectively.  
\end{lemma}

\begin{proof} 
It is straightforward to show that $\mathcal{SL}_m$ is a  well-defined cogroupoid. Next, we verify that $\Phi_e$ is a well-defined character on $\mathcal{SL}_m(e)=\mathcal{SL}_m(e,e)$. Suppose $V$ is of dimension $k$. By \eqref{eq:alg}, in $\mathcal{SL}_m(e)$ we have
\begin{align*}
    \Phi_e\left(\sum_{1\leq i_1,\ldots,i_m\leq k}{e}_{i_1 \cdots i_m}a_{i_1j_1}\cdots a_{i_mj_m}\right)&= \sum_{1\leq i_1,\ldots,i_m\leq k}{e}_{i_1 \cdots i_m}\mathbb P_{i_1j_1}\cdots \mathbb P_{i_mj_m}=e_{j_1\cdots j_m},\\
    \Phi_e\left(\sum_{1\leq i_1,\ldots,i_m\leq k}{e}_{i_1\cdots i_m}b_{i_mj_m}\cdots b_{i_1j_1} \right)&=\sum_{1\leq i_1,\ldots,i_m\leq k}{e}_{i_1\cdots i_m}\mathbb P^{-1}_{i_mj_m}\cdots \mathbb P^{-1}_{i_1j_1}={e}_{j_1\cdots j_m},\\
    \Phi_e(\mathbb{A}\mathbb{B}) &=\mathbb P\mathbb P^{-1}=\mathbb{I}_{k\times k}.
\end{align*}
Finally, by \Cref{lem:antipode}, in $\mathcal{SL}_m(e,f)$ we have
\begin{align*}
    S_{f,e}\circ S_{e,f}(\mathbb A^{e,f})&=\mathbb P^{-1}\mathbb A^{e,f}\mathbb Q=(\Phi_e^{-1}\otimes \id \otimes \Phi_f)(\id \otimes \Delta^{f}_{e,f})(\Delta^{e}_{e,f}\otimes \id)(\mathbb A^{e,f})\\
    &=(\Phi_e^{-1}*\id*\Phi_f)(\mathbb A^{e,f}),\\
        S_{f,e}\circ S_{e,f}(\mathbb B^{e,f})&=\mathbb P^{-1}\mathbb B^{e,f}\mathbb Q=(\Phi_e^{-1}\otimes \id\otimes \Phi_f)(\id \otimes \Delta^{f}_{e,f})(\Delta^{e}_{e,f}\otimes \id)(\mathbb B^{e,f})\\
        &=(\Phi_e^{-1}*\id*\Phi_f)(\mathbb B^{e,f}). 
\end{align*}
It follows that the cogroupoid $\mathcal{SL}_m$ is copivotal.
\end{proof}

\begin{lemma}
\label{prop:SLpreregular}
Let $e$ be an $m$-preregular form on some $n$-dimensional $\kk$-vector space $V$. In the pivotal tensor category ${\rm comod}_{\rm fd}(\mathcal{SL}_m(e))$, the $\mathcal{SL}_m(e)$-comodule map $e: V^{\otimes m}\to \kk$ is preregular.  
\end{lemma}

\begin{proof}
We show that $e$ satisfies both conditions in \Cref{defn:preop} to be preregular.

For (1): we define the $H$-comodule map $\pi: V^{\otimes (m-1)}\to \!^*V$ to be
\[
\pi:=\left(V^{\otimes (m-1)}\xrightarrow{{\rm coev}\otimes \id} (\!^*V\otimes V)\otimes V^{\otimes (m-1)}\xrightarrow{\cong} \!^*V\otimes (V^{\otimes m})\xrightarrow{\id \otimes e} \!^*V\right)
\]
or $\pi(v_{i_1}\otimes \cdots \otimes \cdots v_{i_{m-1}})(v_i)=e(v_1\otimes v_{i_1}\otimes \cdots \otimes v_{i_{m-1}})$. By \Cref{defn:preregular}(1),
one sees that $\pi$ is surjective. Set $W={\rm Ker}(\pi)$, which is a $H$-subcomodule of $V^{\otimes (m-1)}$. Thus $e$ factors through $V\otimes (V^{\otimes (m-1)}/W)$ via $\overline{e}: V\otimes (V^{\otimes (m-1)}/W)\to \kk$. Note that $\pi$ induces an isomorphism of $H$-comodules $\overline{\pi}: V^{\otimes (m-1)}/W\xrightarrow{\sim} V^*$. Set $\widetilde{e}=\overline{e}\circ (\id_V\otimes 
\overline{\pi}^{-1})$. Then we obtain a commutative diagram:
    \[
    \xymatrix{
    V^{\otimes m}\ar[r]^-{e}\ar@{->>}[d] & \kk\\
    V\otimes (V^{\otimes (m-1)}/W)\ar@{->>}[r]^-{\id_V\otimes \overline{\pi}}\ar[ur]^-{\overline{e}} & V\otimes \!^* V\ar[u]_-{\widetilde{e}}.
    }
    \]
It is clear that the $H$-colinear map $\widetilde{e}: V\otimes \!^*V\to \kk$ is nondegenerate. 

For (2): let $\mathbb{P}$ be the invertible matrix associated to $e$. By definition, we have 
\begin{align*}
    D_V^m(e) &(v_{i_1} \otimes \cdots \otimes v_{i_m})\\
    &=({\rm ev})\circ (\varphi_V\otimes \id)\circ (\id\otimes e\otimes \id)\circ ({\rm coev}\otimes \id)(v_{i_1}\otimes \cdots \otimes v_{i_m})\\
    &=({\rm ev})\circ (\varphi_V\otimes \id)\circ (\id\otimes e\otimes \id)\left(\sum_{1\leq i\leq n} v^i\otimes v_i\otimes v_{i_1}\otimes \cdots \otimes v_{i_m}\right)\\
     &= ({\rm ev})\circ (\varphi_V\otimes \id) \left(\sum_{1\leq i\leq n} e(v_i\otimes v_{i_1}\otimes \cdots v_{i_{m-1}}) v^i\otimes  v_{i_m}\right)\\
     \overset{\eqref{eq:cosovereign}}&{=} ({\rm ev}) \left(\sum_{1\leq i,j\leq n} e(v_i\otimes v_{i_1}\otimes \cdots v_{i_{m-1}}) \Phi_e(h_{ij})v^j\otimes  v_{i_m}\right)\\
     &=\sum_{1\leq i,j\leq n} e(v_i\otimes v_{i_1}\otimes \cdots v_{i_{m-1}}) \Phi_e(h_{ij})\delta_{ji_m}\\
    &=\sum_{1\leq i\leq n} e_{ii_1\cdots i_{m-1}}\mathbb P_{ii_m}\\
     \overset{\eqref{eq:preregular}}&{=} e_{i_1\cdots i_m}\\
    &= e(v_{i_1}\otimes \cdots \otimes v_{i_m}),
\end{align*}
for any $1\leq i_1,\ldots,i_m\leq n$. Hence, $e: V^{\otimes m}\to \kk$ is preregular.
\end{proof}

We point out from the above proof that if an $H$-comodule map $e: V^{\otimes m}\to \kk$ is preregular in ${\rm comod}_{\rm fd}(H)$, then it is an $m$-preregular form on $V$ with matrix $\mathbb P\in \GL_n(\kk)$ given by $\varphi_V: \!^*V\to V^*: v^i\mapsto \sum_{1\leq j\leq n}\mathbb P_{ij}v^j$, where $\{v^1,\ldots,v^n\}$ is the dual basis of $\{v_1,\ldots,v_n\}$ of $V$.

Let $H$ be an arbitrary Hopf algebra and $\sigma$ be a left 2-cocycle on $H$. We denote by 
\[F: {\rm comod}_{\rm fd}(H)\xrightarrow{\sim} {\rm comod}_{\rm fd}(H^{\sigma})\] the monoidal equivalence between the categories of finite-dimensional right comodules over $H$ and $H^\sigma$, respectively. Suppose ${\rm comod}_{\rm fd}(H)$ has a pivotal structure $\varphi: \!^*(-)\xrightarrow{\sim} (-)^*$. Thus $\varphi$ induces a pivotal structure $\varphi_\sigma$ on ${\rm comod}_{\rm fd}(H^\sigma)$ via the monoidal equivalence $F$, which is uniquely determined by the commutative diagram 
\begin{align}\label{eq:sovereign}
\xymatrix{
\!^*F(V)\ar[d]_-{(\varphi_\sigma)_{F(V)}}\ar[r]^-{r_V} & F(\!^*V)\ar[d]^-{F(\varphi_V)}\\
F(V)^*\ar[r]^-{l_V} & F(V^*)
}
\end{align}
for any finite-dimensional right $H$-comodule $V$. Here $r: \!^*F(-)\xrightarrow{\sim} F(\!^*(-))$ and $l: F(-)^*\xrightarrow{\sim} F((-)^*)$ are isomorphisms of monoidal functors described in \cite[Eq. (2.8.4)]{Bichon2001}. 

The following definition introduces the notion of twisting of a preregular form via the monoidal equivalence between comodule categories. 

\begin{defn}
\label{DefPreregularForm}
Let $m\ge 2$ be an integer and $V$ be a finite-dimensional right comodule over a Hopf algebra $H$.  For any $H$-comodule map $f: V^{\otimes m} \to \kk$, we define an $H^\sigma$-comodule map $f_\sigma: F(V)^{\otimes_\sigma m}\to \kk$ via 
\begin{align}\label{eq:twistform}
    f_\sigma:=\left(F(V)^{\otimes_\sigma m}\xrightarrow{\xi^{-1}_{V^{\otimes m }}} F(V^{\otimes m})\xrightarrow{F(f)} F(\kk)=\kk\right).
\end{align}
\end{defn}

\begin{proposition}
\label{prop:preregular}
Denote by $\mathcal C={\rm comod}_{\rm fd}(H)$ and $\mathcal D={\rm comod}_{\rm fd}(H^{\sigma})$ together with pivotal $\varphi_\sigma$ satisfying 
\eqref{eq:sovereign}. If $f: V^{\otimes m}\to \kk$ is preregular in $\mathcal C$, then $f_\sigma: F(V)^{\otimes_\sigma m}\to \kk$ is preregular in $\mathcal D$. 
\end{proposition}

\begin{proof} 
To show that $f_\sigma$ is nondegenerate, we consider the commutative diagram
\[
\xymatrix{
F(V)^{\otimes_\sigma m}\ar[rrrr]^-{f_\sigma}\ar[dddd]_-{\cong }\ar[dr]^-{\xi^{-1}_{V^{\otimes m}}} &   &    & & \kk\\
 & F(V^{\otimes m})\ar[rr]^-{F(f)}\ar@{->>}[d] && \kk\ar[ur]^-{=}    &\\
 & F(V\otimes V^{\otimes (m-1)})\ar@{->>}[rr]^-{F(\id_V\otimes \pi)}\ar[d]_-{\xi_{V,V^{\otimes (m-1)}}}&&    F(V\otimes \!^*V)\ar[u]_-{F(\widetilde{f})} &\\
  & F(V)\otimes_\sigma F(V^{\otimes(m-1)})\ar@{->>}[rr]^-{\id_{F(V)}\otimes F(\pi)}\ar[dl]_-{\id_{F(V)}\otimes \xi_{V^{\otimes m}}} &&    F(V)\otimes_\sigma F(\!^*V)\ar[u]_-{\xi^{-1}_{V,\!^*V}} &\\
F(V)\otimes_\sigma F(V)^{\otimes_\sigma (m-1)}\ar[rrrr]^-{\id_{F(V)}\otimes (l_V^{-1}\circ F(\pi)\circ \xi^{-1}_{V^{\otimes (m-1)}})}&  &    & & F(V) \otimes_\sigma \!^*F(V)\ar[ul]_-{r_V}\ar[uuuu]_-{\widetilde{f_\sigma}}
}
\]
where $\widetilde{f_\sigma}:=F(\widetilde{f})\circ \xi^{-1}_{V,\!^*V}\circ r_V$. It is clear that \[r_V^{-1}\circ F(\pi)\circ \xi^{-1}_{V^{\otimes (m-1)}}: F(V)^{\otimes_\sigma (m-1)}\to \!^*F(V)\] is surjective. Since $\widetilde{f}: V\otimes \!^*V\to \kk$ is a nondegenerate pairing in $\mathcal C$, we have a morphism $g: \kk\to \!^*V\otimes V$ such that 
\[
(\widetilde{f}\otimes \id_V)\circ (\id_V\otimes g)=\id_V\quad \text{and}\quad (\id_{\, \!^*V} \otimes \widetilde{f})\circ(g\otimes \id_{\, \!^*V})=\id_{\, \!^*V}.
\]
By \cite[Lemma 1.5]{TV}, we know $\widetilde{f_\sigma}$ is again a nondegenerate pairing in $\mathcal D$ together with morphism 
\[
\kk\xrightarrow{F(g)} F(\!^*V\otimes V)\xrightarrow{\xi_{\!^*V,V}} F(\!^*V)\otimes_\sigma F(V)\xrightarrow{r_V^{-1}\otimes {\id}_{F(V)}} \!^*F(V)\otimes_\sigma F(V).
\]
Moreover, we have $D_{F(V)}^m(f_\sigma)=f_\sigma$ from \Cref{lem:conjugate} since $D_V^m(f)=f$. This proves our result.
\end{proof}

We now state our main result in this section to describe how a preregular form and its associated cogroupoid behave under 2-cocycle twists.

\begin{thm}\label{thm:SLE}
Let $m\ge 2$ be an integer and $V$ be a finite-dimensional $\kk$-vector space. Let $e$ be an $m$-preregular form on $V$ and $\sigma$ be a left 2-cocycle on $\mathcal{SL}_m(e)$. Then $e_\sigma$ is also a $m$-preregular form on $V$ and  
\[
\mathcal{SL}_m(e_\sigma)~\cong~\mathcal{SL}_m(e)^\sigma\] 
as Hopf algebras. 
\end{thm}

In order to prove the above theorem, we first state some results, without proof, for 2-cocycle twists. Parts (3) and (4) appear in the dual context of Drinfeld twists in \cite[Lemma 2.7]{FKMW-Semisimple-reflection-Hopf-algebras-of-dim-16}.

\begin{lemma}
\label{rem:2cocycle}
Let $H$ and $K$ be two arbitrary Hopf algebras with a Hopf algebra map $\alpha: H\to K$. Let $\sigma$ be a left 2-cocycle on $K$ and $B$ be a right $K$-comodule algebra via $\nu: B\to B\otimes K$. 
\begin{enumerate}
    \item We have a left 2-cocycle on $H$, which is denoted by
    \[\alpha^*\sigma:=\left( H\otimes H\xrightarrow{\alpha\otimes \alpha}K\otimes K\xrightarrow{\sigma} \kk\right)
    \]
    with convolution inverse $\alpha^*\sigma^{-1}$. 
    \item The Hopf algebra map $\alpha: H\to K$ induces a Hopf algebra map $\alpha^\sigma:H^{\alpha^*\sigma}\to  K^\sigma$ for the 2-cocycle twists of the corresponding Hopf algebras, where $\alpha=\alpha^{\sigma}$ as linear maps.  
    \item $\sigma^{-1}$ is a left 2-cocycle on $K^\sigma$ with $(K^\sigma)^{\sigma^{-1}}=K$. Similarly, $\alpha^*\sigma^{-1}$ is a left 2-cocycle on $H^{\alpha^*\sigma}$ with $(H^{\alpha^*\sigma})^{\alpha^*\sigma^{-1}}=H$.
    In particular, we have $(\alpha^{\sigma})^{\sigma^{-1}}=\alpha: H\to K$.
    \item $B_{\sigma^{-1}}$ is a right $K^\sigma$-comodule algebra via $\nu^{\sigma}: B_{\sigma^{-1}}\to B_{\sigma^{-1}}\otimes K^{\sigma}$, where $\nu^{\sigma}=\nu$ as linear maps. In particular, $(B_{\sigma^{-1}})_\sigma=B$ and $(\nu^{\sigma})^{\sigma^{-1}}=\nu$. 
    \item Suppose $B$ is also a right $H$-comodule algebra via $\tau: B\to B\otimes H$ such that the 
    diagram \[
    \xymatrix{
    &B\ar[dl]_-{\tau}\ar[dr]^-{\nu}  &\\
   B\otimes H\ar[rr]^-{\id \otimes \alpha}&& B\otimes K
    }
    \]
    commutes. Then $B_{\sigma^{-1}}$ is also a right $H^{\alpha^*\sigma}$-comodule algebra via $\tau^{\sigma}: B_{\sigma^{-1}}\to B_{\sigma^{-1}}\otimes H^{\alpha^*\sigma}$ satisfying $\nu^{\sigma}=(\id\otimes \alpha^{\sigma})\, \circ \, \tau^{\sigma}$, where $\tau^{\sigma}=\tau$ as linear maps. 
    \end{enumerate}
\end{lemma}

\begin{proof}[Proof of Theorem \ref{thm:SLE}] 
By \Cref{prop:SLpreregular} and \Cref{prop:preregular} with $H= \mathcal{SL}_m(e)$, we know $e$ and $e_\sigma$ are preregular in ${\rm comod}_{\rm fd}(H)$ and in ${\rm comod}_{\rm fd}(H^{\sigma})$, respectively. Then $e_\sigma$ is an $m$-preregular form on $F(V)=V$. Recalling \Cref{lemma:univ-sl}, we write the universal right Hopf coactions of $\mathcal{SL}_m(e)$ and $\mathcal{SL}_m(e_\sigma)$ on $V$ and $F(V)$ as
\[
\rho: V\to V\otimes \mathcal{SL}_m(e) \quad \text{and} \quad \tau: F(V)\to F(V)\otimes \mathcal{SL}_m(e_\sigma),
\]
such that $e: V^{\otimes m}\to \kk$ and $e_\sigma: F(V)^{\otimes_\sigma m}\to \kk$ are colinear with respect to the corresponding coactions. Since $\mathcal{SL}(e)^\sigma$ left coacts on $F(V)$ via $\rho^\sigma: F(V)\to F(V)\otimes_\sigma \mathcal{SL}_m(e)^\sigma$ and $e_\sigma$ is $\mathcal{SL}_m(e)^\sigma$-colinear, we have a unique Hopf algebra map $h: \mathcal{SL}_m(e_\sigma)\to \mathcal{SL}_m(e)^\sigma$ such that 
\[
(\id_{F(V)}\otimes h)\circ \tau=\rho^\sigma.
\]

It follows from \Cref{rem:2cocycle} that $\sigma^{-1}$ is a left 2-cocycle on $\mathcal{SL}_m(e)^\sigma$ which induces a left 2-cocycle $h^*\sigma^{-1}$ on $\mathcal{SL}_m(e_\sigma)$. Furthermore, one checks that $\mathcal{SL}_m(e_\sigma)^{h^*\sigma^{-1}}$ left coacts on $V$ via $\tau^{h^*\sigma^{-1}}: V\to V\otimes  \mathcal{SL}_m(e_\sigma^{-1})^{h^*\sigma^{-1}}$ while preserving the map $e: V^{\otimes m}\to \kk$. So we get another unique Hopf algebra map $g: \mathcal{SL}_m(e)\to \mathcal{SL}_m(e_\sigma)^{h^*\sigma^{-1}}$ such that
\[
(\id_{V}\otimes g)\circ \rho=\tau^{h^*\sigma^{-1}}.
\]
We denote by $h^{\sigma^{-1}}: \mathcal{SL}_m(e_\sigma)^{h^*\sigma^{-1}}\to \mathcal{SL}_m(e)$ the Hopf algebra map induced by $h: \mathcal{SL}_m(e_\sigma)\to \mathcal{SL}_m(e)^\sigma$. Thus, the commutative diagram \[
\xymatrix{
&V\ar[dl]_-{\rho}\ar[dr]^-{\rho}\ar[d]^-{\tau^{h^*\sigma^{-1}}}&\\
V\otimes \mathcal{SL}_m(e)\ar[r]_-{{\id}\otimes g} & V\otimes \mathcal{SL}_m(e_\sigma)^{h^*\sigma^{-1}}\ar[r]_-{\id\otimes h^{\sigma^{-1}}} & V\otimes \mathcal{SL}_m(e)
}
\]
implies that
\[
h^{\sigma^{-1}}\circ g~=~\id\ \text{on}\ \mathcal{SL}_m(e).
\]
Similarly, we have $(h^{\sigma^{-1}})^*\sigma=h^*\sigma$ is a left 2-cocycle on $\mathcal{SL}_m(e_\sigma)^{h^*\sigma^{-1}}$. Note that $g^*(h^{\sigma^{-1}})^*\sigma=(h^{\sigma^{-1}}\circ g)^*\sigma=\sigma$ on $\mathcal{SL}_m(e)$. Therefore, we have a Hopf algebra map
\[
g^{(h^{\sigma^{-1}})^*\sigma}:\mathcal{SL}_m(e)^\sigma\to \mathcal{SL}_m(e_\sigma),
\]
making the diagram
\[
\xymatrix{
&V\ar[dl]_-{\tau}\ar[dr]^-{\tau}\ar[d]^-{\rho^\sigma}&\\
F(V)\otimes \mathcal{SL}_m(e_\sigma)\ar[r]_-{{\id}\otimes h} & V\otimes \mathcal{SL}_m(e)^\sigma\ar[r]_-{\id\otimes g^{(h^{\sigma^{-1}})^*\sigma}} & F(V)\otimes \mathcal{SL}_m(e_\sigma)
}
\]
commute. This implies that
\[
g^{(h^{\sigma^{-1}})^*\sigma}\circ h=\id\ \text{on}\ \mathcal{SL}_m(e_\sigma).
\]
Finally, it is routine to check that $h$ and $g^{(h^{\sigma^{-1}})^*\sigma}$ are inverse to each other.
\end{proof}

\bibliography{biblNov2021}
\bibliographystyle{amsplain}
\end{document}